\documentclass[11pt]{amsart}
\usepackage{geometry}                % See geometry.pdf to learn the layout options. There are lots.
\usepackage{graphicx}
\usepackage{amssymb}
\usepackage{epstopdf}
\usepackage{cancel}

\usepackage{amsmath,amscd}

\usepackage{xypic}

\usepackage{rotating}

\usepackage{multirow, array}

\usepackage{tikz}
\usetikzlibrary{matrix,arrows}

\usepackage{mathtools}

\usepackage{hyperref}

\title{The Structure of Masses of rank $n$ Quadratic Lattices of varying determinant 
over number fields}
\author{Jonathan Hanke}
\date{\today}                                           % Activate to display a given date or no date

\begin{document}
\maketitle

\newtheorem{thm}{Theorem}[section]  
\newtheorem{lem}[thm]{Lemma}  
\newtheorem{cor}[thm]{Corollary}  
\newtheorem{defn}[thm]{Definition}  
\newtheorem{rem}[thm]{Remark}  
\newtheorem{conj}[thm]{Conjecture}
\newtheorem{question}[thm]{Question}

\newtheorem{sublem}[thm]{Sublemma}

\newcommand{\Mat}{\mathrm{Mat}}
\renewcommand{\O}{\mathrm{O}}
\newcommand{\SO}{\mathrm{SO}}
\newcommand{\SL}{\mathrm{SL}}
\newcommand{\GL}{\mathrm{GL}}
\newcommand{\ord}{\mathrm{ord}}
\newcommand{\Gen}{\mathrm{Gen}}
\newcommand{\Sym}{\mathrm{Sym}}
\newcommand{\sgn}{\mathrm{sgn}}
\newcommand{\SqCl}{\mathrm{SqCl}}
\newcommand{\Mass}{\mathrm{Mass}}
\newcommand{\Aut}{\mathrm{Aut}}
\newcommand{\Disc}{\mathrm{Disc}}
\newcommand{\DiscSq}{\mathrm{DiscSq}}
\newcommand{\Pic}{\mathrm{Pic}}
\newcommand{\Gal}{\mathrm{Gal}}
\newcommand{\Vol}{\mathrm{Vol}}

\newcommand{\N}{\mathbb{N}}
\newcommand{\Z}{\mathbb{Z}}
\newcommand{\Q}{\mathbb{Q}}
\newcommand{\R}{\mathbb{R}}
\newcommand{\C}{\mathbb{C}}
\newcommand{\F}{\mathbb{F}}
\renewcommand{\S}{\mathbb{S}}
\newcommand{\T}{\mathbb{T}}
\newcommand{\U}{\mathbb{U}}
\renewcommand{\O}{\mathbb{O}}
\newcommand{\E}{\mathbb{E}}
\newcommand{\A}{\mathbb{A}}

\newcommand{\V}{\mathbb{V}}
\newcommand{\res}{\mathrm{res}}
\newcommand{\Nm}{\mathrm{Norm}}
\newcommand{\Image}{\mathrm{Image}}

\newcommand{\x}{\vec x}
\newcommand{\y}{\vec y}
\renewcommand{\c}{\vec c}

\newcommand{\M}{\mathcal{M}}
\renewcommand{\H}{\mathcal{H}}

\renewcommand{\[}{\left[}
\renewcommand{\]}{\right]}
\renewcommand{\(}{\left(}
\renewcommand{\)}{\right)}

\newcommand{\ra}{\rightarrow}
\newcommand{\al}{\alpha}
\newcommand{\ve}{\varepsilon}

\newcommand{\leg}[2]{\(\frac{#1}{#2}\)}

\newcommand{\p}{\mathfrak{p}}
\newcommand{\q}{\mathfrak{q}}
\newcommand{\OF}{\mathcal{O}_F}
\newcommand{\OK}{\mathcal{O}_K}
\newcommand{\Op}{\mathcal{O}_\p}
\newcommand{\Oq}{\mathcal{O}_\q}
\newcommand{\Ov}{\mathcal{O}_v}

\newcommand{\Supp}{\mathrm{Supp}}
\newcommand{\Cond}{\mathrm{Cond}}     %% DEFINE THE CONDUCTOR!!!

\newcommand{\n}{\mathfrak{n}}
\newcommand{\s}{\mathfrak{s}}
\renewcommand{\v}{\mathfrak{v}}
\newcommand{\g}{\mathcal{G}}

\newcommand{\I}{\mathfrak{I}}

\renewcommand{\P}{\mathcal{P}}

%% Some Squareclass Shortcuts
\newcommand{\SqAfInt}{\SqCl(\A_{F, \mathbf{f}}^\times, U_\mathbf{f})}
\newcommand{\SqAf}{\SqCl(\A_{F, \mathbf{f}}^\times)}

\newcommand{\SqFInt}{\SqCl(F^\times, \OF^\times)}
\newcommand{\SqF}{\SqCl(F^\times)}

\newcommand{\SqFpInt}{\SqCl(F_\p^\times, \Op^\times)}
\newcommand{\SqFp}{\SqCl(F_\p^\times)}

\newcommand{\SqFvInt}{\SqCl(F_v^\times, \Ov^\times)}
\newcommand{\SqFv}{\SqCl(F_v^\times)}

\newcommand{\SqOp}{\SqCl(\Op^\times)}
\newcommand{\SqOF}{\SqCl(\OF^\times)}

\section{Overview and Notation}

\subsection{Introduction}

The notion of the ``mass'' of a positive definite integral quadratic form $Q$ was first introduced as such by Smith in his 1867 paper \cite{Smith:1867} and also by Minkowski in his 1885 dissertation \cite{Mink}, though special cases can be found in earlier works of Gauss, Dirichlet and Eisenstein.
In this setting, the {\bf mass} of $Q$ is defined to be the rational number 
$$
\Mass(Q) := \sum_{[Q'] \in \Gen(Q)} \frac{1}{|\Aut(Q')|}  \in \Q > 0
$$
given by summing the reciprocals of the sizes of the automorphism groups $\Aut(Q')$ of all $\Z$-equivalence classes of quadratic forms $Q'$ (denoted by $[Q']$) where $Q'$ is equivalent to $Q$ by some invertible linear change of variables over $\Z/m\Z$ for each $m \in \N$ and also over the real numbers $\R$.  (The set of such $Q'$ is called the {\bf genus} $\Gen(Q)$ of $Q$, and it is well-known that this sum is finite.)
The mass is an interesting quantity because it is closely related to the number of classes in the genus (called the {\bf class number} $h(Q)$ of $Q$), but also because it can be computed analytically as an infinite product over all primes $p \in \N$.

In 1935-1937,  Siegel \cite{Si1, Si2, Si3, Siegel:1963vn} revolutionized the analytic theory of quadratic forms by providing a general framework for understanding previous mass formulas, as well as exact formulas for representing numbers by certain quadratic forms.  All of these formulas have a very ``volume-theoretic" character, and express certain weighted sums over classes $[Q'] \in \Gen(Q)$ as a product over all places $v$ (of our given number field, e.g. $\Q$) as a product of ``local densities'' which measure the ``volume of solutions'' of some given quadratic diophantine equation.
In particular Siegel's Mass formula states that 
$$
\Mass(Q) = 2\prod_v \beta_{Q,v}(Q)^{-1}
$$
where the $\beta_{Q,v}(Q)$ are the local densities (for $Q$ representing itself) at the place $v$.
This perspective was used by Tamagawa \cite{Tamagawa:1966ys, Weil:1982zr} in the 1960's to give a new proof of Siegel's formula for $\Mass(Q)$ in terms of a canonical (Tamagawa) measure on the associated adelic special orthogonal group $\SO_\A(Q)$.

To use the mass formula to evaluate $\Mass(Q)$ for any given quadratic form $Q$ involves evaluating the local density factors at all places $v$, which can be rather complicated at the primes $p \mid 2\det(Q)$. These factors have been worked out in various cases by many authors, though the state of the literature about these formulas is not entirely adequate.  In particular, a correct formula at the place $p=2$ is given in \cite{CS} without proof, and this appears not to agree with the proven formula \cite{Watson:1976dq}.  For $p \neq 2$ correct formulas are given in \cite{Pall, CS}, and though the analogous formulas holds at any prime ideal $\p\nmid 2$ of a number field $F$, a reference for this is difficult to find (perhaps in \cite{Pfeuffer:1971pd}?).

\medskip 
One very natural question 
%about the mass formula 
that has not received much attention is how to understand the {\bf total mass} $\mathrm{TMass}_n(S)$ of positive definite $\Z$-valued quadratic forms in $n$ variables with (Hessian) determinant $S$, defined as 
$$
\mathrm{TMass}_n(S) := \sum_{[Q] \text{ with} \det(Q) = S} \frac{1}{|\Aut(Q)|},
$$
and how $\mathrm{TMass}_n(S)$ varies with $S$.  These are the main questions that we address in this paper, for any fixed $n \geq 2$.

From the perspective of Siegel's mass formula, the total mass $\mathrm{TMass}_n(S)$ is a somewhat less natural quantity to study than $\Mass(Q)$ since it involves summing masses of quadratic forms across all genera of a given determinant, giving a complicated sum of complicated infinite products.  However, rather remarkably, this summation has the effect of smoothing out much of the variation of the mass formula, and allows us to give a formula for the total mass $\mathrm{TMass}_n(S)$ for any given determinant $S$.  
Since it is well-known that the variation of the archimedean local density is given by $\beta_{Q,\infty}(Q)^{-1} = C_n(Q) \cdot S^{\frac{n(n-1)}{2}}$ for some constant $C_n(Q)$, we instead study the {\bf total non-archimedean mass} $T_n(S)$ defined by 
$$
T_n(S) := \beta_\infty(Q) \cdot \mathrm{TMass}_n(S) 
=
\sum_{\substack{[Q] \text{ with } \\ \det(Q) = S}} 2 \cdot \prod_p \beta_p(Q)^{-1},
$$
and the {\bf primitive total non-archimedean mass}
%its primitive variant 
$T^*_n(S)$ defined by instead summing over classes $[Q]$ of primitive integer-valued forms with $\det(Q) = S$.

%%%%%%%%%%%%%%%%%%%%%%%%%%%%%%%%%%%%%%%%%%%%%%%%%

Our main results, stated for simplicity in the special case of positive definite forms over $F=\Q$ with $n$ odd, show that 

\begin{thm}
When $n$ is odd, the formal Dirichlet series 
$$
D_{T^*; n}(s) := \sum_{S\in\N} \frac{T^*_n(S)}{S^{s}}
$$
can be written as a sum
$$
D_{T^*; n}(s) = \kappa_n \cdot \[ D_{A^*; n}(s) + D_{B^*; n}(s) \] %\qquad\qquad \text{for some explicit constant $K_n$}
$$
of two Eulerian Dirichlet series $D_{A^*; n}(s)$ and $D_{B^*; n}(s)$, with some explicit constant $\kappa_n$  
(Corollaries \ref{Cor:Formal_Dirichlet_M} and \ref{Cor:Formal_Dirichlet_decomp_n_odd}).
\end{thm}
\noindent
When $n$ is even this sum of Dirichlet series is slightly more complicated, and naturally gives such a decomposition where the overall constant $\kappa_n$ depends on the squareclass $t\N^2$ that $S$ lies within.
We also show that 
\begin{thm}
The Euler factors at $p$ in the Dirichlet series $D_{A^*; n}(s)$ and $D_{B^*; n}(s)$ above are each 
rational functions in $p^{-s}$ (Theorem \ref{Thm:Rationality_of_Euler_factors}, Corollary \ref{Cor:Rationality_of_AB_Euler_factors}).
\end{thm}
\noindent 
When $n=2$ we explicitly compute these Euler factors at all primes $p$ (Theorem \ref{Thm:Explicit_A_and_B_for_n=2}). Finally, we use these explicit local computations to recover Dirichlet's class number formula for imaginary quadratic fields $K/\Q$ (Theorem \ref{Thm:Analytic_class_number_formula}), and discuss some similarities between the total non-archimedean mass series $D_{T; n}(s)$ and modular forms of half-integral weight.

\smallskip 
We actually prove our main theorems over a general number field $F$, which causes us to make several technical distinctions that are blurred when $F=\Q$.  Over $F$ we see that not all lattices are free, (leading us to replace quadratic forms by the more general notion of quadratic lattices), the discriminant becomes an integral squareclass (or more accurately, a non-archimedean tuple of local integral squareclasses indexed by the prime ideals $\p$ of $F$), and there is no natural notion of how to associate a squareclass to an ideal (leading us to define the notions of a ``formal squareclass series'' and of a ``family of distinguished squareclasses'').  Also, substantial technical work must be done to ``normalize'' the local densities within a squareclass over $F$ (where over $\Q$ we have a natural notion of the ``squarefree part'' of a number and a simpler local theory of genera) and to establish an analytic class number formula generalizing Dirichlet's class number formula to the setting of CM-extensions $K/F$.  
%This generalization has a somewhat unexpected factor $Q_{K/F}$ which is either 1 or 2.   
All of our main results include the freedom to discuss quadratic lattices of any fixed signature (which is now a vector), and also allow one to specify their Hasse invariants at finitely many primes $\p$.

\smallskip
One important motivation for studying the total mass of given determinant comes from the arithmetic implications of the ``discriminant-preserving" correspondences introduced by Bhargava \cite{Bh1, Bh2, Bh3, Bh4} that generalize Gauss composition for binary quadratic forms.
Also, masses of certain ternary quadratic forms summed across several genera of fixed determinant (called $S$-genera) were studied in \cite{Berk-H-Jagy} to establish a family of ``$S$-genus identities'', though there the determinants were essentially squarefree.
A forthcoming paper \cite{Hanke_n_equals_3_masses} explicitly computes the local Euler factors for $D_{A^*; n}(S)$ and $D_{B^*; n}(S)$ when $n=3$, in preparation for an investigation of the growth of the $2$-parts of class groups of cubic fields \cite{Bh-Ha-Sh} jointly with M. Bhargava and A. Shankar.

\medskip
{\bf Acknowledgements:} The author would like to thank Manjul Bhargava for posing a question that led to this work, and MSRI for their hospitality during their Spring 2011 semester in Arithmetic Statistics.  The author would also like to warmly thank Robert Varley for his continuing interest in this work, and for several helpful conversations.
This work was completed at the University of Georgia between December 2009 and Summer 2011, and was partially supported by the NSF Grant DMS-0603976.

\subsection{Detailed Summary}

In this paper we study the ``primitive total non-archimedean mass'' $T^*_{\vec{\sigma}_\infty, \c_{\S}; n}(S)$ of rank $n$ quadratic lattices over a number field $F$ of a fixed Hessian determinant squareclass $S$, signature vector $\vec{\sigma}_\infty = (\sigma_v)$ for all real archimedean places of $F$, and specified Hasse invariants $c_\p$ at some finite set $\S$ of primes $\p$ of $F$. 
Our main interest is in understanding how the total non-archimedean mass varies as we vary its determanant squareclass $S$, and its behaviour as ``$S \rightarrow \infty$'' in various ways.
%
%
%%%%%%%%%%%
%{\bf (ADD TAUBERIAN STUFF???)}
%Our interest is in how this quantity grows with $S$, and we can give precise asymptotic information about various weighted ``Tauberian-type'' sums $\sum_S S^\al M_{S, \vec{\sigma}_\infty, \c_{\S}; n}$ with $\al \in \R>0$ over these quantities behave.
%%%%%%%%%%%%

In \textsection2 we define local, global, and adelic notions of (integral and rational) squareclasses associated to genera of quadratic lattices, and describe some related structures and normalizations.  We also go to some effort to show that there is at most one local genus of integral quadratic forms associated to any normalized determinant squareclass, and to characterize the squareclasses that arise from quadratic lattices.

In \textsection3 we define and study the total non-archimedean mass, and show that it can be studied formally as a purely local object.

In \textsection4 we introduce the notion of a formal squareclass series to encode how the total non-archimedean mass varies with its determinant squareclass.  In this language we show that the associated formal squareclass series is a precise linear combination of two formal squareclass series admitting Euler product expansions, and determine how this linear combination depends on the signature and chosen Hasse invariants.
One of these terms is independent of our fixed signature and Hasse invariant conditions, while the other oscillates (as a sum or difference) depending on them.  We regard the first term as the ``main term'', and the second as the ``error term'' or ``secondary term''.  This is particularly interesting because one usually regards local densities via the Siegel-Weil formula as themselves contributing the ``main term'' of the theta series, so the fact that these exhibit further structure is a very interesting feature.  
We also give a way of associating a (family of) formal Dirichlet series to these formal squareclass series, and prove a similar structure theorem in that context.  
Once this structural result is established, the main question is to understand the rationality of the Euler factors of each of these series.  For this we give a purely local formulation of these Euler factors as a weighted sum over local genera of quadratic forms of with a fixed determinant squareclass.   This formulation allows one in principle to use the theory of $\p$-adic integral invariants for quadratic forms and $\p$-adic local density formulas to explicitly compute these factors in any case of interest, though such computations rapidly become non-trivial.

In \textsection5 we use aspects of the theory of local genus invariants to show that the Euler factors previously considered can be written as rational functions, and when $n$ is odd we show that these Euler factors can be fully understood as we vary our local normalized squareclass (and that this dependence is almost constant).

In \textsection6 we establish a precise connection between our total non-archimedean mass and the mass of quadratic lattices defined as a weighted sum over classes.  When the class number $h(\OF)=1$  we show that formal Dirichlet series made from these masses decompose nicely.

Finally, in \textsection7 we perform an case-by-case analysis with local genus invariants to exactly compute the Euler factors whose rationality was previously established.  At primes $\p\mid 2$ we use the 
``train/compartment'' formalism of Conway \cite[p381]{CS-book} to describe the $2$-adic genus invariants, and the (stated but not proven) mass formula of Conway and Sloane \cite{CS} to compute the local density $\beta_{Q,\p }(Q)^{-1}$ for primes $\p\mid2$.  Because of this, our computations are valid for any number field $F$ where $p=2$ splits completely.
%
%%%%%%
To treat number fields with different splitting behavior at $p=2$ by these methods one must give both a good theory of integral invariants and a formula for computing local densities.  There is a rather complicated theory of dyadic integral invariants that has been worked out by O'Meara \cite{OMeara:1955rr, OMeara:1957cr}, but these papers have received little attention in the literature.
The theory of explicit dyadic masses for arbitrary quadratic forms has not been fully worked out.

We conclude in \textsection8 by using results of Kneser and of \textsection7 to establish an analytic class number formula for $h(\OK)$, where $K/F$ is a CM extension of number fields and $p=2$ splits completely in $F$.  We also remark how this formula can be generalized to allow quadratic orders, and show it is compatible with the more traditional class number formula arising from ratios of Dedekind zeta functions.

\subsection{Notation}

Throughout this paper we let $\Z := \{\cdots, -2, -1, 0, 1, 2, \cdots\}$ denote the integers,  $\Q$ the rational numbers, $\R$ the real numbers, $\C$ the complex numbers, and $\N$ the natural numbers (i.e. positive integers).  We also denote by $\Z_{\geq 0}$ the non-negative integers.  
We denote the units (i.e. invertible elements) of a ring $R$ by $R^\times$, and let $\mathrm{char}(K)$ denote the characterisitic of a field $K$.
We also let $\Mat_{m \times n}(R)$ denote the ring of $m \times n$ matrices over $R$, and set $\GL_n(R) := \Mat_{n \times n}(R)^\times$.
For any object, we let the subscript $\bullet$ refer to an unspecified set of extra parameters for the object.
We write $A \equiv B_{(m)}$ to mean that the elements/sets given by $A$ and $B$ are equal in the ring $R/mR$, where the ambient ring $R$ is implicitly known.

\smallskip
{\bf Number Fields:}
We let $F$ denote a number field with ring of integers $\OF$, $v$ is a place of $F$, $\p$ is a prime ideal (or more simply, a {\bf prime}) of $F$, $F_v$ is the completion of $F$ at $v$, $\Ov$ is the ring of integers of $F_v$ (which is $F_v$ itself when $v$ is archimedean).  
We let $[F:\Q]$ denote the absolute degree of $F$ (giving its dimension as a $\Q$-vectorspace) and let $\Delta_F \in \Z$ denote the absolute discriminant of $F$.
For a prime $\p$, we denote its {\bf residue field} at $\p$ by $k_\p := \Op/\p\Op$, which is a finite field of size $q := |\Op/\p\Op|$.
We denote by $\infty$ the archimedean place of the rational numbers $\Q$, and write $v\mid\infty$ (resp. $v\mid\infty_\R$) to denote that $v$ is archimedean (resp. real archimedean).  We also identify the two conjugate embeddings $v$ where $F_v = \C$. We denote the set of non-archimedean (finite) places (i.e. primes $\p$) of $F$ by $\mathbf{f}$.  
For any finite set $\T \subset \mathbf{f}$ we let $I^{\T}(\OF)$ denote the set of invertible (integral) ideals of $\OF$, relatively prime to all $\p\in\T$.  We also adopt the general convention that quantities denoted by Fraktur letters (e.g. $\p, \mathfrak{a}, \n, \s, \v$, etc.) will be (possibly fractional) ideals of $F$.

We denote by $\A_{F, \mathbf{f}}^\times := \prod'_{\p \in \mathbf{f}} F_\p^\times$ the non-archimedean ideles of $F$, where the restricted direct product $\prod'$ requires that all but finitely many components lie in $\Op^\times$.  For convenience we will often write products $\prod_{\p\in\mathbf{f}}$ more simply as unquantified products $\prod_\p$, and similarly write $\prod_v$ for a product over all places $v$ of $F$.

When $F = \Q$, we define the {\bf conductor} $\Cond(\chi) \in \N$ of a Dirichlet character $\chi: (\Z/m\Z)^\times \rightarrow \C^\times$, as the smallest $f \in \N$ so that $\chi$ factors through a character on 
$(\Z/f\Z)^\times$.  We also say that $D \in \Z$ is a {\bf fundamental discriminant} if $D$ is the absolute discriminant of a quadratic number field.

We refer to an extension of number fields $K/F$ of degree $[K:F]=2$ where $F$ is totally real and $K$ is totally complex as a {\bf CM extension}, where CM here is an abbreviation for ``complex multiplication''.

\smallskip
{\bf Squareclasses:}
For any abelian group $G$ with subgroup $H$, we let $\SqCl(G,H):= G / (H^2)$ denote the group of {\bf $H$-squareclasses of $G$}, and write $\SqCl(G) := \SqCl(G,G)$ for the group of {\bf squareclasses of $G$}.
We will frequently refer to the group of {\bf (integral) non-archimedean squareclasses} 
$\SqAfInt$ where $U_\mathbf{f} := \prod_\p \Op^\times$, and let $S_\p$ denote the component of $S$ at $\p$.  
We define the {\bf valuation at $\p$} of  a non-archimedean squareclass $S$ by the expression $\ord_\p(S) =  \ord_\p(S_\p) := \ord_\p(\al)$ when  $S_\p = \al(\Op^\times)^2$ for some $\al \in F_\p^\times$, and we
define its {\bf support} $\Supp(S)$ 
as the set of primes $\p$ of $F$ where $\ord_\p(S) \neq 0$.

\smallskip
{\bf Basic Quadratic Objects:}  
For any $n \in \N$ we define an $n$-dimensional {\bf quadratic space} over a field $K$ (assuming that $\mathrm{char}(K) \neq 2$) to be a pair $(V, Q)$ where $V$ is an $n$-dimensional vectorspace 
(which we assume is equipped with a basis $\mathcal{B}$) 
over $K$ and $Q$ is a quadratic form on $V$ (relative to $\mathcal{B}$).  
Given a quadratic form $Q(x_1, \cdots, x_n)$ we define its Hessian and Gram matrices respectively as $H = (h_{ij}) :=  (\frac{\partial^2 Q}{\partial x_i \partial x_j})$ and Gram matrix $G := \frac{1}{2}H$.
We respectively define the Gram and Hessian determinants $\det_G(Q)$ and $\det_H(Q)$ of the quadratic space $(V,Q)$ as the squareclass in $\SqCl(K^\times)$ given by taking the determinant of the Gram and Hessian matrices of $Q$.  We say that $(V, Q)$ is {\bf non-degenerate} if the Gram determinant $\det_G(Q) \neq 0$ (which happens iff $\det_H(Q) \neq 0$ since $\mathrm{char}(K)\neq 2$).

We say that $L$ is a {\bf quadratic $R$-lattice} for some ring $R$  if $R$ is a subring of $F$,  $L$ is a finitely generated $R$-submodule of some quadratic space $(V,Q)$ over $F$, and $L \otimes_R F = V$.  Note that a quadratic lattice $L$ inherits the values of its ambient quadratic space, and for any set $\T$ we say that $L$ is {\bf $\T$-valued} if $Q(L) \subset \T$. We say that a quadratic ($\OF$- or $\Op$-)lattice $L$ in {\bf primitive} if $L$ is ($\OF$- or $\Op$-)valued and any scaling $Q \mapsto c \cdot Q$ of the ambient quadratic space for which 
$L$ is still ($\OF$- or $\Op$-)valued must have $c \in$ ($\OF$ or $\Op$).

Given a quadratic $\OF$-lattice $L$, we denote its localizations $L_v := L \otimes_{\OF} \Ov$ in the quadratic spaces $(V_v := \otimes_{F} F_v, Q_v)$ over $F_v$.  We let $\s(L)$ denote its {\bf scale ideal} (generated locally by the entries of its Gram matrix in a basis for $L_\p$), $\n(L)$ its {\bf norm ideal} (generated by its values), its {\bf volume ideal} $\v(L)$ (generated locally by $\det_G(L_\p)$ relative to a basis for $L_\p$), and its {\bf norm group} $\mathcal{G}(L)$ (generated by its values).

\smallskip
{\bf Decorated Quadratic Objects:}  %We let $Q$ denote a quadratic form in $n$-variables.
We define %the sets 
$\mathbf{Gen}^*(S, \vec{\sigma}_\infty, \c_{\S}; n)$, $\mathbf{Cls}^*(S, \vec{\sigma}_\infty, \c_{\S}; n)$ and $\mathbf{Cls}^{*,+}(S, \vec{\sigma}_\infty, \c_{\S}; n)$ respectively to be the set of genera, classes, and proper classes $G$ of primitive $\OF$-valued rank $n$ quadratic $\OF$-lattices  with fixed signature $\sigma_v(G) = (\vec{\sigma}_\infty)_v$ at all places $v \mid \infty$, fixed Hasse invariants $c_\p(G) = (\c_\S)_\p$ at the finitely many primes $\p\in\S$, and Hessian determinant $\det_H(G) = S$.

We define $\mathbf{Gen}^*_{\p}(S, \ve; n)$ as the set of local genera $G_\p$ of primitive $\Op$-valued rank $n$ quadratic forms with Hessian determinant $\det_H(G_\p) = S$ and Hasse invariant $c_\p(G_\p) = \ve$.
We also let 
$\mathbf{Gen}^*_{\mathbf{f}}(S, \ve_\infty, \c_{\S}; n)$ be the set of tuples $(G_\p)$ of 
$G_\p \in \mathbf{Gen}^*_{\p}(S_\p, \{\pm1\}; n)$ over all primes $\p$ 
where we require that
$\det_H(G_\p) \in \Op^\times$ for all but finitely many primes $\p$, 
and also $\prod_\p c_\p(G_\p) = \ve_\infty$.  (Note that here $c_\p(G_\p) = 1$ for all but finitely many $\p$ since the Hilbert symbol $(x,y)_\p = 1$ when $\p\nmid 2$ and $x,y\in \Op^\times$.)

%%%%%%%%%%%%%%%%%%%%%%%%%%%%%%%%%%%%%%%%%%%%%%%%%%%
%%%%%%%%%%%%%%%%%%%%%%%%%%%%%%%%%%%%%%%%%%%%%%%%%%%
%%%%%%%%%%%%%%%%%%%%%%%%%%%%%%%%%%%%%%%%%%%%%%%%%%%

%%%%%%%%%%%%%%%%%%%%%%%%%%%%%%
%% =========================================
%%  Facts about local Quadratic forms
%% =========================================
%%%%%%%%%%%%%%%%%%%%%%%%%%%%%%

\section{Facts about local genera of integer-valued quadratic forms}

In this section we are interested in defining and understanding the Hessian determinant squareclasses of (non-degenerate) $\OF$-valued quadratic $\OF$-lattices.  We classify their associated ideals, show locally that when the associated ideal is maximal there is a unique genus of quadratic forms giving rise to this squareclass, and finally establish exactly which squareclasses arise from quadratic lattices.
In later sections these observations will allow us to define certain ``normalized'' local densities, and to sensibly parametrize Hessian determinant squareclasses by integral ideals.

%%%%%%%%%%%%%%%%%%%%%%%%%%%%%%%%%%%%%%%%%
%%  Definitions of Hessian determinant squareclass and Rational Reduction    %%
%%%%%%%%%%%%%%%%%%%%%%%%%%%%%%%%%%%%%%%%%

\begin{defn}
Given a quadratic $\OF$-lattice $L$, we define its {\bf (non-archimedean) Hessian determinant squareclass} $\det_H(L)$ as the squareclass $S \in \SqCl(\A_{F, \mathbf{f}}^\times, U_\mathbf{f})$ so that $S_\p = \det_H(Q_\p)$ where $Q_\p$ is the quadratic form associated with the quadratic lattice $L_\p$ with respect to some choice of independent generators  for $L_\p$ as a free $\Op$-module.  Note that $S_\p$ does not depend on this choice of generators, and that $\ord_\p(S_\p)=0$ for all but finitely many primes $\p$.
\end{defn}

In order to better understand the Hessian determinant squareclass $\det_H(L)$, we first study the relationships between local and global squareclasses under both rational and integral equivalence.

\begin{lem}
We can express the integral and rational non-archimedean squareclasses as restricted direct products of local integral and rational squareclasses by
$$
\SqAfInt
 	 = \sideset{}{'} \prod_\p \SqFpInt 
$$
$$
\SqAf
	 = \sideset{}{'} \prod_\p \SqFp 
$$
where both restricted direct products $\prod'_p$ are with respect to the family of open subgroups $\SqOp$,
using the inclusion 
$\SqOp  \hookrightarrow \SqFp$
in the second product.
\end{lem}

\begin{proof}
The restricted direct product with respect to the subgroups $\SqOF$ follow from the restricted product defining of $\A_{F, \mathbf{f}}^\times$, and the local squareclass group factors $\SqCl(F_\p^\times, \cdot)$ are obtained by looking at the surjective restriction map to the component at any given prime $\p$, since $U_\mathbf{f} \cap F_\p^\times  = \Op^\times$ and $\A_{F, \mathbf{f}}^\times \cap F_\p^\times  = F_\p^\times$.
\end{proof}

We notice that by passing from integral to rational equivalence gives {\bf rational reduction maps}, denoted by $\rho_{*}$, giving the commutative diagram
\begin{equation}  \label{Eq:Rational_reduction}
\begin{tikzpicture}
[baseline=(current bounding box.center)]
[description/.style={fill=white,inner sep=2pt}]
\matrix (m) [matrix of math nodes, row sep=3em,
column sep=2.5em, text height=1.5ex, text depth=0.25ex]
{ 
\SqFInt &  \SqAfInt \\
\SqF & \SqAf \\
};
\path[->,font=\scriptsize]
%\path[right hook->] 
(m-1-1) edge node[auto] {$\Delta$} (m-1-2)
%\path[right hook->] 
(m-2-1) edge node[auto] {$\Delta$} (m-2-2);
%(m-1-1) edge node[auto] {$ \varphi $} (m-1-2)
%	edge node[description] {$ \Psi $} (m-2-2)
\path[->>] (m-1-2) edge node[auto] {$ \rho_\mathbf{f}$} (m-2-2);
\path[->>] (m-1-1) edge node[auto] {$ \rho$} (m-2-1);
\end{tikzpicture}
\end{equation}
where $\Delta$ denotes the diagonal map $x \mapsto (x, x, \cdots)$.
We also can realize the non-archimedean rational reduction map $\rho_\mathbf{f}$ as the product 
$\rho_\mathbf{f} = \prod_\p \rho_\p$ 
of the local rational reduction maps $\rho_\p: \SqFpInt \twoheadrightarrow \SqFp$.

\begin{defn}
We say that a non-archimedean integral squareclass $S \in \SqCl(\A_{F, \mathbf{f}}^\times, U_\mathbf{f})$ is {\bf globally rational} if its rational reduction $\rho_\mathbf{f}(S)$ is in the image of $\SqF$ in (\ref{Eq:Rational_reduction}).
\end{defn}

%%%%%%%%%%%%%%%%%%%%%%%%%%%%%%%%%%
%%  Study ideal structure of Hessian determinant squareclasses    %
%%%%%%%%%%%%%%%%%%%%%%%%%%%%%%%%%%
We begin by studying the possible valuations attained by the Hessian determinant squareclasses, which we phrase in the language of ideals.

\begin{defn}
For any non-archimedean integral squareclass $S \in  \SqCl(\A_{F, \mathbf{f}}^\times, U_\mathbf{f})$ 
we define its {\bf associated (valuation) ideal} $\mathfrak{I}(S)$ by 
$$
\mathfrak{I}(S) := \prod_\p \p^{\ord_\p(S_\p)}.
$$
This product is finite since $S_\p \subseteq \Op^\times$ for all but finitely many primes $\p$.
\end{defn}

\begin{lem} \label{Lem:Hessian_det_ideals}
Suppose that $L$ is a rank $n$  $\OF$-valued quadratic $\OF$-lattice.  Then 
the associated ideal 
$\I(\det_H(L)) \subseteq \OF$ and 
for each $n$,
the sum $\mathfrak{h}_n$ of all such ideals
is %given by
\begin{align*}
\mathfrak{h}_n 
:= 
\biggl\{ 
\textstyle{\prod_\p \det_H}(L_\p) \Op 
\biggm\vert
\begin{aligned}
&\text{$L_\p$ is an $\Op$-valued } \\
&\text{rank $n$ quadratic form} 
\end{aligned}
\biggr\} 
=
\begin{cases}
\OF  & \quad \text{if $n$ is even,} \\ 
2\OF  & \quad \text{if $n$ is odd.} \\ 
\end{cases}
\end{align*}
\end{lem}

\begin{proof}
The localization $L_\p:= L\otimes_{\OF} \Op$ of $L$ is a free $\Op$-module, and so by taking a basis for $L_\p$ we can represent it by an $\Op$-valued quadratic form.  
By taking an $\Op$-basis of $L_\p$, we obtain an $\Op$-valued quadratic form whose Hessian matrix is in $\Mat_{n\times n}(\Op)$ with even diagonal, hence $\det_H(L_\p) \in \Op$.  The Leibniz  formula 
$$
\det(A) = \sum_{\sigma\in S_n} \sgn(\sigma) \prod_{i=1}^n a_{i \sigma(i)}
$$
for a square matrix $A = (a_{ij})$ with even diagonal tells us that the terms from permutations $\sigma$ with no fixed points determine how $\mathfrak{h}_n$ sits between $\OF$ and $2\OF$.   When $A$ is symmetric the terms from $\sigma$ and $\sigma^{-1}$ are the same, hence they contribute twice when $\sigma \neq \sigma^{-1}$.  
Since a fixed-point-free involution $\sigma \in S_n$ exists $\iff n$ in even, we see that $\mathfrak{h}_n  = 2\OF$ when $n$ in odd.
If $n$ is even then we can choose the involution $\sigma(i) := n-i + 1$ and define the matrix $A = (a_{ij}) :=(\delta_{i \sigma(i)})$ which has $\det(A) \in \{\pm1\}$, so $\mathfrak{h}_n  = \OF$.
\end{proof}

\begin{thm}[Uniqueness of Normalized Local Genera] \label{Thm:Uniqueness_for_normalized_QF}
Suppose that $n \in \N$ and that $S_\p \in \SqCl(F_\p^\times, \Op^\times)$ with $\ord_\p(S_\p) = \ord_\p(\mathfrak{h}_n)$.  Then there exists at most one local genus $G_\p$ of primitive $\Op$-valued quadratic forms in $n$ variables with (Hessian) determinant $\det_H(G_\p) = S_\p$. 
\end{thm}

\begin{proof}
The condition that $G_\p$ is $\Op$-valued is equivalent to the local norm ideal $\n(G_\p)$ being integral, and 
the ideal generated by the Hessian determinant squareclass is the volume ideal $\v(2G_\p)$ (defined in \cite[\textsection82:9, p221]{OM}).  The volume ideal $\v(L)$ can be expressed in terms of any Jordan splitting $L = \bigoplus_{i \in \Z} L_i$ as
$$
\v(L) = \p^{
\sum_{i \in \Z} 
i \cdot \mathrm{rank}_{\Op}(L_i) 
}.
$$
We also have that  
$$
\s(L) = \p^{
\min\{i \in \Z \,\mid\, \mathrm{rank}_{\Op}(L_i) \neq 0\}
}
$$
and the scale and norm ideals are related by the containments
$\s(G_\p) \supseteq \n(G_\p) \supseteq 2\s(G_\p)$.

When $\p\nmid 2$ this shows that $\s(G_\p) = \n(G_\p) = \Op$ and the largest possible ideal $\v(2L)$ is $\Op$, which is attained by the unimodular forms.  When $\p\nmid 2$ unimodular forms exist for every $n \in \N$, and they are exactly characterized up to isomorphism by their determinant squareclass \cite[\textsection92:2, p247]{OM}.

When $\p\mid 2$ the condition $2\s(G_\p) \subseteq \n(G_\p) = \Op$ implies that $\s(2G_\p) \subseteq \Op$ and so $\v(2G_\p)\subseteq \Op$.  If $\v(2G_\p) = \Op$ then $2G_\p$ is unimodular and satisfies $\n(2G_\p) \subsetneq \s(2G_\p)$, so by \cite[\textsection93:15, p 258]{OM} we see that $2G_\p$ is an orthogonal direct sum of  binary forms, and so $n$ in even.
If $n$ is odd then largest volume ideal is given by $\v(2G_\p) = 2\Op$ because we can take a rank $n-1$ unimodular form that must be a direct sum of binary forms of norm ideal $2\Op$, and then since $\n(2G_\p) = 2\Op$ the remaining unary form $\al x^2$ must have $\s(\al x^2) = \n(\al x^2) \subseteq 2\Op$, so the volume ideal is largest when $\v(2G_\p) = \s(\al x^2) = 2\Op$ (which can be attained).  We now examine each of these two possibilities.
%
%\smallskip
%\rule{\linewidth}{0.5mm}
%\rule{6in}{0.5mm}
%\rule{5in}{0.2mm}
\centerline{\rule{5in}{0.1mm}}
%\smallskip

%\medskip
{\bf Case 1 ($n$ even):} When $n$ is even, we define 
$$
\mathcal{L}_{\p; \text{ even}} 
:= 
\mathcal{L}_{\p; n} 
:= 
\left\{
\begin{tabular}{c}
Local genera of \\
non-degenerate  rank $n$ \\
quadratic $\Op$-lattices $L := L_0$ \\
\end{tabular}
\middle|
\begin{tabular}{c}
$L_0$ is unimodular \\
with $\n(L_0) = 2\Op$
\end{tabular}
\right\}.
$$

\begin{sublem}
Suppose that $n\in \N$ is even and $L, L' \in \mathcal{L}_{\p; n}$. Then
$$
\textstyle
\det_H(L) = \det_H(L') 
\qquad \Longrightarrow \qquad
L \cong L'.
$$
\end{sublem}

\begin{proof}[Proof of Sublemma]
We first analyze the Hasse invariants of the unimodular non-diagonalizable binary quadratic lattices with $\n(L) = 2\Op$.  (From \cite[\textsection93:15, p258]{OM} and $\n(L_0) = 2$ we know that $L_0$ must be a sum of lattices of this form.)  We know from \cite[\textsection93:11, pp255-6]{OM} we know that the only such lattices (up to equivalence) are $A(0,0)$ and $A(2, 2\rho)$ where $A(\al, \beta) := \al x^2 + 2xy + \beta y^2$, $\rho \in \SqOp$ is in the squareclass defined by the relation $\Omega =: 1 + 4\rho$ 
and $\Omega\in \SqOp$ is the unique squareclass with quadratic defect $4\Op$.   

{\bf Case a)}  Suppose that $Q = A(0,0) = 2xy$ over $\Op$.  Then $Q \sim_{F_\p} x^2 - y^2$, giving $\det_H(Q) = (-1)(\Op^\times)^2$ and $c_\p(Q) = (1,-1)_\p = 1$.

{\bf Case b)}  Suppose that $Q = A(2, 2\rho) = 2x^2 + 2xy + 2\rho y^2$ over $\Op$.  Then $Q \sim_{F_\p} 2x^2 + (2\rho - \frac{1}{2})y^2$, giving $\det_H(Q) = (4\rho - 1)(\Op^\times)^2 = (\Omega - 2)(\Op^\times)^2$ and $c_\p(Q) = (2, 2\rho - \frac{1}{2})_\p = (2, 2\cdot (4\rho - 1))_\p = \cancel{(2,2)_\p} \cdot (2, \Omega - 2)_\p$.

We note here that the two determinant squareclasses $\det_H(Q)$ are the two possible lifts of the (mod 4) squareclass $-1 \in \SqCl((\Op/4\Op)^\times)$.

From \cite[\textsection93:18(ii), p260]{OM} (and since there $\ord_\p(a) + \ord_\p(b) = \ord_\p(2) + \ord_\p(2)$ is even) we see that any $L \in \mathcal{L}_{\p, even}$ can be written as a direct sum of copies of $A(0,0)$ and at most one $A(2, 2\rho)$, the presence/absence of $A(2, 2\rho)$ is determined by $\det_H(L)$.  Therefore the lattice $L$ is completely determined by $\det_H(L)$ and $n$.
\end{proof}
\vspace{-.2in}
\centerline{\rule{5in}{0.1mm}}

{\bf Case 2 ($n$ is odd):} When $n \in \N$ is odd, we let
$$
\mathcal{L}_{\p; n} 
:= 
\left\{
\begin{tabular}{c}
Local genera of \\
non-degenerate  rank $n$ \\
quadratic $\Op$-lattices \\
$L := L_0 \oplus L_1$ \\
\end{tabular}
\middle|
%\begin{tabular}{c}
%$L_0$ in unimodular of rank $n-1$ with $\n(L_0) = 2\Op$ \\
%and $L_1 =: 2u x^2$ is $2\Op$-modular with $\n(L_1) = 2\Op$\\
%\end{tabular}
\begin{tabular}{c}
$L_0$ is unimodular of rank $n-1$, \\
$L_1 =: 2u x^2$ is $2\Op$-modular, \\
and $\n(L_0) = \n(L_1) = 2\Op$ \\
\end{tabular}
\right\}.
$$
Notice that if $G' \in \mathcal{L}_{\p; n}$ then $G'$ is $2\Op$-valued and so $\frac{1}{2}G'$ is a primitive $\Op$-valued genus with normalized Hessian determinant $\det_H(\frac{1}{2}G')$.
From above, we also see that $\mathcal{L}_{\p; n}$ contains all local genera $2G_\p$ where $G_\p$ is a primitive $\Op$-valued local genus and $\det_H(G_\p)$ is normalized.
This shows that if $G_\p$ is a primitive $\Op$-valued local genus then 
$$
\text{$\textstyle\det_H(G_\p)$ is normalized $\Longleftrightarrow G_\p \in \mathcal{L}_{\p; n}$.}
$$

\begin{sublem}
Suppose that $n\in \N$ is odd and $L, L' \in \mathcal{L}_{\p; n}$. Then
$$
\textstyle
\det_H(L) = \det_H(L') 
\qquad \Longrightarrow \qquad
V \cong V',
$$
where $V$ and $V'$ are the quadratic spaces associated to $L$ and $L'$ respectively.
%%%%%
\end{sublem}

\begin{proof}[Proof of Sublemma]
When $n$ is odd then the decomposition $V = V_0 \oplus V_1$ has local invariants 
%\begin{align*}
$$
\textstyle
\det_H(V) = \det_H(V_0) \cdot \det_H(V_1), 
$$
%\\
%
%\quad \text{ and } \quad 
$$
\textstyle
c_\p(V) = c_\p(V_0) \cdot c_\p(V_1) \cdot (\det_G(V_0), \det_G(V_1))_\p.
$$
%\end{align*}
We know that $\det_H(V_1) = u \in \SqOp$ and $c_\p(V_1) = 1$ since $V_1 = 2u x^2$, 
and our previous computations show that $\det_H(V_0) = (-1)^\frac{n-1}{2} \Omega^\epsilon$ with $c_\p(V_0)$ uniquely determined by $\epsilon \in \{0,1\}$.
This gives the local invariants of $V$ as
$$
\textstyle
\det_H(V) = (-1)^\frac{n-1}{2} \Omega^\epsilon u
\qquad
\text{ and }
\qquad
c_\p(V) = c_\p(V_0) \cdot ((-1)^\frac{n-1}{2} \Omega^\epsilon, 2u)_\p.
$$

Now suppose that $L, L' \in \mathcal{L}_{\p, n}$ with $n$ odd, and that their associated quadratic spaces have direct sum decompositions as 
$V = V_0 \oplus V_1$ and $V' = V'_0 \oplus V'_1$,
%%%%
%$$
%V = V_0 \oplus V_1 \qquad \text{ and } \qquad V' = V'_0 \oplus V'_1 ,
%$$
where $V_i$ and $V'_i$ are the quadratic spaces associated to $L_i$ and $L'_i$ respectively.
%
%$V_i = L_i \otimes_{\Op} F_\p$ and $V'_i = L'_i \otimes_{\Op} F_\p$.  
If $\det_H(V) = \det_H(V')$ then knowing that there are exactly two squareclasses in $\SqOp$ with given reduction in $\{\pm 1\} \subseteq \SqOp{(\Op/4\Op)^\times}$, 
we have that either $\epsilon' = \epsilon$ and $u' = u$ (giving $V' \cong V$), or 
\begin{equation}
\epsilon' \neq \epsilon \quad\text{ and }\quad u' = u \cdot \Omega.
\end{equation}
In the second case we check that $c_\p(V') = c_\p(V)$ since by \cite[\textsection63:11a, p165]{OM} we have that 
$$
((-1)^\frac{n-1}{2} \Omega, 2u)_\p 
= -((-1)^\frac{n-1}{2}, 2u)_\p
= ((-1)^\frac{n-1}{2}, 2u\cdot \Omega)_\p,
$$
and so again $V'\cong V$.
This shows that if $L, L' \in \mathcal{L}_{\p, n}$ with $n$ odd and $\det_H(L) = \det_H(L')$, then  their associated quadratic spaces are equivalent. %$V \cong V'$.
\end{proof}

\begin{sublem}
Suppose that $n\in \N$ is even and $L, L' \in \mathcal{L}_{\p; n}$. Then
$$
\textstyle
\det_H(L) = \det_H(L') 
\qquad \Longrightarrow \qquad
L \cong L'.
$$
\end{sublem}

\begin{proof}[Proof of Sublemma]
To finish the proof, we compute O'Meara's ideal $\mathfrak{f}_0(L)$ (which is ideal $\mathfrak{f}$ associated to the unimodular lattice $L_0$) and the weight ideals $\mathfrak{w}_i(L)$, then use the dyadic integral equivalence conditions in \cite[\textsection93:28, pp267-8]{OM}.  Since $u_i = \ord_\p(\s(L_i)) = \ord_\p(2)$ for $i \in \{0,1\}$ (so $u_0 + u_1 \in 2\Z$), the defining equation for $\mathfrak{f}_0$ on \cite[p264]{OM} becomes
$$
\Op^2 \cdot \mathfrak{f}_0 
= \sum_{\al \in 2\Op, \beta \in 2u\Op^2}
	\mathfrak{d}(\al\beta) + \underbrace{2\p^{\frac{2 \ord_\p(2)}{2} + 0}}_{=\, 4\Op}.
$$
Since the quadratic defect satisfies $\mathfrak{d}(\alpha^2 \cdot \xi) = \alpha^2\mathfrak{d}(\xi)$ (\cite[p160]{OM}), we see that the sum above is $4\sum_{\gamma \in \Op^\times} \mathfrak{d}(\gamma) \subseteq 4\Op$, so $\mathfrak{f}_0 = 4\Op$.  We also compute that $\mathfrak{w}_0(L) = 2\Op$ and $\mathfrak{w}_1(L) = \Op$ from the known norms and scales.  With these, we see that the respective conditions in \cite[\textsection93:28, pp267-8]{OM} for $L, L' \in \mathfrak{L}_{\p; n}$ to be equivalent are:
\begin{enumerate}
\item $\det_H(L_0) \equiv \det_H(L'_0) \pmod {4\Op}$,
\item There is an isometry $V_0 \rightarrow V$,
\item no condition.
\end{enumerate}
When $\det_H(L) = \det_H(L')$ we have that $V \sim V'$ so the second condition certainly holds, and the first condition follows from our previous analysis since $\Omega \equiv 1 \pmod {4\Op}$.
This shows that $L\cong L'$, proving the theorem.
\end{proof}

\vspace{-.2in}
\centerline{\rule{5in}{0.1mm}}

These sublemmas together show that there is at most one local genus of lattices of any given normalized local Hessian determinant.
\end{proof}

%%%%%%%%%%%%%%%%%%%%%%%%%%%%%%%%%%%%%%
%%  Describe exactly the possible Hessian determinant squareclasses    %%
%%%%%%%%%%%%%%%%%%%%%%%%%%%%%%%%%%%%%%

\begin{lem}[Classifying Hessian squareclasses]  \label{Lem:Globally_rational}
Given $n \in\N$, then a non-archimedean squareclass $S \in \SqAfInt$ is the Hessian determinant squareclass of a some non-degenerate rank $n$ (possibly not $\OF$-valued) quadratic $\OF$-lattice $L$ iff $S$ is globally rational.
\end{lem}

\begin{proof}
If $L \subset (V,Q)$ is a non-degenerate quadratic $\OF$-lattice then its Hessian determinant $\det_H(L) =: S = \prod_\p S_\p$ where $S_\p := \det_H(Q'_\p)$ is the Hessian determinant of a quadratic form $Q'_\p$ obtained by choosing $n$ independent generators for $L_\p$ as an $\Op$-module (though $S_\p$ doesn't depend on this choice).  Since these generators for $L_\p$ also give a basis for the local quadratic space $(V_\p, Q_\p)$, we see that 
$S_\p = \det_H(Q_\p) \in \SqFp$ and so $S_\p = (\det_H(Q))_\p$ where $\det_H(Q) \in \SqF$,
which shows that $\det_H(L)$ is always globally rational.

Now suppose that $S \in \SqCl(\A_{F, \mathbf{f}}^\times, U_\mathbf{f})$ is globally rational.  Then $S$ must differ from the squareclass $\al(\Op^\times)^2$ for some $\al\in F^\times$ at only finitely many primes, and denote the set of such primes by $\T$.  
At each $\p\in\T$ we have that $S_\p / \al \in (F_\p^\times)^2$, so we can find some $\beta_\p \in F_\p^\times$ so that $S_\p / \al = (\beta_\p)^2$.
By the Hasse principle for quadratic spaces \cite[\textsection63:23, p171]{OM} we can find an $n$-dimensional quadratic space $(V, Q)$ over $F$ with $\det_H(Q) = \al$.
 We also have that $\al =\det_H(L)$ for the standard free lattice $L := (\OF)^n$ (all relative to a given fixed basis for $V$).
By choosing for each $\p\in \T$ some $\lambda_\p \in \GL_n(F_\p)$ with $\det(\lambda_\p) = \beta_\p$, we see that the local lattices $L'_\p := \lambda_\p L_\p$ have $\det_H(L'_\p) = S_\p$ and these can be assembled (together with $L_\p$ for all $\p\notin\T$) to a quadratic lattice $L' \subset (V, Q)$ with $\det_H(L) = S$.  This shows that every globally rational non-archimedean squareclass arises as $\det_H(L)$.
\end{proof}

\begin{rem} \label{Rem:Local_genus_existence_when_2_splits_completely}
From the proof of Theorem \ref{Thm:Uniqueness_for_normalized_QF} we see that a local genus $G_\p$ with $\det_H(G_\p) = S_\p$ exists for any normalized squareclasss $S_\p \in \SqFpInt$ when $\I(S_\p) \subseteq \mathfrak{h}_n \Op$ both when $n$ is odd, and when $n$ is even and $\p\nmid 2$.  When $n$ is even and $F_\p = \Q_2$, we have existence iff either $\I(S_\p) \subseteq \Op$ and $S \equiv (-1)^\frac{n}{2}$ or $\I(S_\p) \subseteq 4\Op$.
\end{rem}

%%%%%%%%%%%%%%%%%%%%%%%%%%%%%%%
%% -----------------------------------------------------------------------------
%%  End of Facts about local Quadratic forms
%% -----------------------------------------------------------------------------
%%%%%%%%%%%%%%%%%%%%%%%%%%%%%%%

%%%%%%%%%%%%%%%%%%%%%%%%%%%%%%%%%%%%%%%%%%%%%%%%%%%
%%%%%%%%%%%%%%%%%%%%%%%%%%%%%%%%%%%%%%%%%%%%%%%%%%%
%%%%%%%%%%%%%%%%%%%%%%%%%%%%%%%%%%%%%%%%%%%%%%%%%%%

%%%%%%%%%%%%%%%%%%%%%%%%%%%%%%
%% =========================================
%%  Total Non-archimedean Mass
%% =========================================
%%%%%%%%%%%%%%%%%%%%%%%%%%%%%%

\section{The total non-archimedean mass}

In this section we define our main object of study, the ``primitive total non-archimedean mass'' $T^*$ of primitive integer-valued quadratic lattices of rank $n$ with given signature, Hasse invariants (at finitely many primes), and Hessian determinant squareclass over a number field $F$.  We show that this primitive total non-archimedean mass can be naturally considered as a purely local object $M^*$, and we examine its convergence properties in this local context.

\begin{defn}  \label{Def:Signature_vector}
Given a number field $F$ and some $n \in \N$ we let $\vec{\sigma}_\infty$ denote a {\bf vector of signatures} for $F$ of rank $n$, by which we mean a vector $\vec{\sigma}_\infty := ((\sigma_{v, +}, \sigma_{v, -}))_{\{v \mid \infty_\R\}} \in \prod_{v \mid \infty_\R} (\Z_{\geq 0} \times \Z_{\geq 0})$ where for each real archimedean place $v$ of $F$ we have $\sigma_{v, +} + \sigma_{v, -}= n$.
\end{defn}

\begin{rem}
Notice that one obtains a vector of signatures of rank $n$ naturally from any $n$-dimensional non-degenerate quadratic space $(V, Q)$ over $F$, and that this vector determines the product of (non-degenerate) quadratic spaces $\prod_{v\mid\infty} (V_v, Q_v)$ up to isomorphism.
\end{rem}

\begin{defn} \label{Def:Hasse_vector}
Let $\S$ be a finite set of non-archimedean places of a number field $F$ and define a {\bf vector of Hasse invariants}, or {\bf Hasse vector}, as a vector $\c_\S \in \prod_{\p\in \S} \{\pm 1\}$.
\end{defn}

\begin{defn} \label{Def:M_rational}
For a fixed number field $F$, $n \in \N$, vectors $\vec{\sigma}_\infty$ and $\c_\S$ of signatures of rank $n$ and Hasse invariants, and some globally rational 
non-archimedean squareclass $S \in \SqAfInt$,
we define the {\bf primitive total rational non-archimedean mass} of (Hessian) determinant $S$ by 
$$
T^*_{\vec{\sigma}_\infty, \c_{\S}; n}(S) := \sum_{G \in \mathbf{Gen}^*(S, \vec{\sigma}_\infty, \c_{\S}; n)} 
\prod_\p \beta_{G, \p}(G)^{-1}.
$$
\end{defn}

To study the global quantity $T_{\vec{\sigma}_\infty, \c_{\S}; n}(S)$ we notice that this can be recovered from a more local quantity that is a little easier to study, though there are some convergence problems for arbitrary products when $n=2$.  Notice that in the following formulation the signature vector $\vec{\sigma}_\infty$ above is replaced by a single sign $\ve_\infty \in \{\pm1\}$, which will help to clarify the dependence of the total mass on the choice of signatures.

\begin{defn} \label{Def:M_idelic}
For a fixed number field $F$,  $n \in \N$, some $\ve_\infty \in \{\pm1\}$, a vector  $\c_\S$ of  Hasse invariants,  and some non-archimedean squareclass $S \in \SqCl(\A_{F, \mathbf{f}}^\times, U_\mathbf{f})$,
we define the {\bf primitive total adelic non-archimedean mass} of (Hessian) determinant $S$ by 
$$
M^*_{\ve_\infty, \c_{\S}; n}(S) := \sum_{G \in \mathbf{Gen}^*_{\mathbf{f}}(S, \ve_\infty, \c_{\S}; n)}  
	\prod_\p \beta_{G, \p}(G)^{-1},
$$
when the infinite product converges, and zero otherwise.
Notice that by Theorem \ref{Thm:Uniqueness_for_normalized_QF} the local genera $G_\p$ are uniquely determined for $\p \notin \Supp(S) \cup \{\p\mid2\}$, so the sum over 
$\mathbf{Gen}^*_{\mathbf{f}}(S, \ve_\infty, \c_{\S}; n)$
is finite.
\end{defn}

We now address some convergence issues for $M^*_{\ve_\infty, \c_{\S}; n}(S)$, and see how it recovers $T^*_{\vec{\sigma}_\infty, \c_{\S}; n}(S)$ as a special case.

%%%%%%%%%%%%%%%
%%  Convergence Lemma   %
%%%%%%%%%%%%%%%
\begin{lem}
The quantity $M^*_{\ve_\infty, \c_{\S}; n}(S)$ converges for all $S\in \SqCl(\A_{F, \mathbf{f}}^\times, U_\mathbf{f})$ when $n\neq 2$.  If $n=2$ and $S$ is globally rational (i.e. agrees with some given element of $\SqF$) 
at all but finitely many places $\p$ then $M^*_{\ve_\infty, \c_{\S}; n}(S)$ converges.
\end{lem}

\begin{proof}
Suppose that $S\in \SqAfInt$.
If 
$\I(S) \nsubseteq \mathfrak{h}_n$
then the sum over $G$ is empty and the empty product converges.   If 
$\I(S) \subseteq \mathfrak{h}_n$ 
then by Theorem \ref{Thm:Uniqueness_for_normalized_QF} we know that for every $\p \notin \T := \Supp(S) \cup \{\p\mid 2\}$ there is a unique (unimodular) local genus $G_\p$ over $\Op$ with $\det_H(G_\p) = S_\p$, and the infinite products converge iff the product $\prod_{\p\notin\T} \beta^{-1}_{G_\p, \p}(G_\p)$ converges.

For $\p \notin \T$ we have by Hensel's Lemma that $\beta_{G_\p, \p}(G_\p) = \frac{|SO_{Q_\p}(k_\p)|}{q^\frac{n(n-1)}{2}}$, and the sizes of orthogonal groups (of types $B$ and $D$) over $k_\p$ are given in \cite[p72]{Carter}.  This leads to the formulas
$$
\beta_{Q_\p, \p}(Q_\p)
=
\begin{cases}
(1-\frac{1}{q^2})(1-\frac{1}{q^4})\cdots (1-\frac{1}{q^{2r}}) & \text{if $n = 2r+1$ is odd,} \\
(1-\frac{1}{q^2})(1-\frac{1}{q^4})\cdots (1-\frac{1}{q^{2r-2}})(1\pm\frac{1}{q^{r}}) & \text{if $n = 2r$ is even,} \\
\end{cases}
$$
where the choice of sign when $n$ is even is given by $\pm = - \chi_\p(S_\p)$ where $\chi_\p$ is the non-trivial quadratic character of $\SqCl(k_\p^\times)$.
This tells us that the generic product $\prod_{\p\notin\T} \beta^{-1}_{G_\p, \p}(G_\p)$ is given by $\zeta_F^\T(2) \zeta_F^\T(4) \cdots \zeta_F^\T(n-1)$ when $n\geq 3$ is odd, and the bounds
$$
1 \leq \prod_{\p\notin\T} \beta^{-1}_{G_\p, \p}(G_\p)
\leq \zeta_F^\T(2) \zeta_F^\T(4) \cdots \zeta_F^\T(n-2) \zeta_F^\T(\tfrac{n}{2}) / \zeta_F^\T(n)
$$
when $n$ is even.  This ensures convergence for arbitrary $S$ when $n\neq 2$ (as $n=1$ has generic factor $1$), but shows that when $n=2$ we can have divergent products if almost all signs are $+$.

However if $n=2$ and $S$ agrees with some $S \in \SqF$ on $\T$ (after possibly enlarging $\T$) then we see that $\prod_{\p\notin\T} \beta^{-1}_{G_\p, \p}(G_\p) = L_F^\T(1, \chi)$ where $\chi$ is the finite order Hecke character of $F$ given by the Legendre symbol $\p \mapsto \leg{a}{\p}$ for all $\p\nmid 2$, which converges.
\end{proof}

\begin{lem} \label{Lem:T_is_M}
Suppose the non-archimedean squareclass $S \in \SqAfInt$ is globally rational.  
Then 
$$
T^*_{\vec{\sigma}_\infty, \c_{\S}; n}(S) = M^*_{\ve_\infty, \c_{\S}; n}(S)
$$
where $\ve_\infty := \prod_{v\mid\infty} c_v((\vec{\sigma}_\infty)_v)$
\end{lem}

\begin{proof}
We will prove this by establishing a bijection
$\mathbf{Gen}^*(S, \vec{\sigma}_\infty, \c_{\S}; n)
\overset{\sim}{\rightarrow}
\mathbf{Gen}^*_{\mathbf{f}}(S, \ve_\infty, \c_{\S}; n)$ 
when $S$ is globally rational, and use the identification of (local or global) genera with the orthogonal group orbits of lattices in a (local or global) quadratic space to describe them.
(See \cite[\textsection1.2 and \textsection4.5]{Ha-AWS} for more details.)

We have a localization map $\lambda: L \mapsto (L_\p)_\mathbf{f}$ that takes a genus of lattices $L$ to a non-archimedean tuple of genera of lattices $L_\p$ in their respectively localized quadratic spaces.  Notice that  $\det_H(L) = (\det_H(L_\p))_\mathbf{f}$, $c_\p(V) = c_\p(V_\p)$, and $\mathrm{rank}_{\OF}(L) = \mathrm{rank}_{\Op}(L_\p) = n$ for all primes $\p$, and the product formula for Hasse invariants gives $\prod_\p c_\p(V) = \prod_{v\mid\infty} c_v(V)$.  Therefore by restricting to primitive and $\OF$-valued genera (which are local properties of $L$), we obtain a localization map 
$\lambda: 
\mathbf{Gen}^*(S, \vec{\sigma}_\infty, \c_{\S}; n)
\rightarrow
\mathbf{Gen}^*_{\mathbf{f}}(S, \ve_\infty, \c_{\S}; n)$.

To see that that the map is injective, notice that if $G_\mathbf{f} := \lambda(G)$ then by the Hasse principle for quadratic spaces there is a unique quadratic space $(V,Q)$ localizing to the quadratic spaces $(V_\p,Q_\p)$ of $G_\mathbf{f}$, which is the quadratic space of $L$ (up to global isomorphism).  Also, because global lattices are uniquely determined by their localizations, we see that $L$ is also determined (up to equivalence), so $\lambda$ is injective.

To see that $\lambda$ is surjective when $S$ is globally rational, first notice that by Lemma \ref{Lem:Globally_rational} any 
$G \in \mathbf{Gen}^*(S, \vec{\sigma}_\infty, \c_{\S}; n)$ must have $S$ globally rational, so the condition is necessary.  However when $S$ is globally rational, then by the Hasse principle \cite[\textsection72:1, p203]{OM} any 
$G_\mathbf{f} \in \mathbf{Gen}^*_{\mathbf{f}}(S, \ve_\infty, \c_{\S}; n)$
will have a unique quadratic space $(V,Q)$ over $F$ that localizes to the quadratic spaces of $(G_\mathbf{f})_\p$ at all primes $\p$, and then the local-global principle for lattices  \cite[\textsection81:14, p218]{OM} these $L_\p$ can be assembled to some $\OF$-lattice on $V$ since $\det_H(G_\mathbf{f}) \in \SqAfInt$ implies that $L_\p = (\Op)^n$ for almost all $\p$.
\end{proof}

%%%%%%%%%%%%%%%%%%%%%%%%%%%%%%%
%% -----------------------------------------------------------------------------
%%  End of Total Non-archimedean Mass
%% -----------------------------------------------------------------------------
%%%%%%%%%%%%%%%%%%%%%%%%%%%%%%%

%%%%%%%%%%%%%%%%%%%%%%%%%%%%%%%%%%%%%%%%%%%%%%%%%%%
%%%%%%%%%%%%%%%%%%%%%%%%%%%%%%%%%%%%%%%%%%%%%%%%%%%
%%%%%%%%%%%%%%%%%%%%%%%%%%%%%%%%%%%%%%%%%%%%%%%%%%%

%%%%%%%%%%%%%%%%%%%%%%%%%%%%%%
%% =========================================
%%  Main definitions of the Local Density Formal series
%% =========================================
%%%%%%%%%%%%%%%%%%%%%%%%%%%%%%

\section{The Structure of certain Formal  Series}

In this section we use the language of formal series to study the primitive total adelic non-archimedean mass $M^*_{\ve_\infty, \c_{\S}; n}(S)$ locally, and
show that certain formal series associated to the total non-archimedean mass (as the determinant squareclass varies) naturally decomposes (along squareclasses of squareclasses) as a linear combination of two series each of which admits an Euler product.  In Section \ref{Sec:Binary_Forms} we will compute these Euler factors explicitly when $n=2$.

\medskip

%%%%%%%%%%%%%%%%%%%%%%%%%%%
%%  Idelic squareclass series and Euler products  %%
%%%%%%%%%%%%%%%%%%%%%%%%%%%
Given a function $X_{\bullet}:\SqCl(\A_{F, \mathbf{f}}^\times) \ra \C$, 
we define the {\bf formal non-archimedean squareclass series} associated to $X_{\bullet}$ as formal sum of the form 
$$
\mathcal{F}_{X, \bullet} := 
\sum_{S \in \SqAfInt} \frac{X_{\bullet}(S)}{S}
$$
where the formal symbols $S$ run over idelic square classes $S \in \SqAfInt$.  If $X_{\bullet}(S)$ is a {\bf multiplicative function} on $\SqAfInt$ (meaning that $X_{\bullet}(S) = \prod_\p X_{\p, \bullet}(S_\p)$ for some functions $X_{\p, \bullet}$ on $\SqFpInt$) that is trivial on the %valuation zero 
local squareclasses $\SqOp$ for all but finitely many primes $\p$ (which is required for the multiplicativity to make sense formally), 
then we can express it as local product 
of the form
$$
%F = 
\sum_{S \in \SqAfInt} \frac{X_{\bullet}(S)}{S} = \prod_\p 
\underbrace{
\sum_{S_\p \in \SqFpInt} 
	\frac{X_{\p, \bullet}(S_\p)}{S_\p}
}_\text{Euler factor at $\p$}
$$
where $X_{\p, \bullet}(S_\p)$ is defined by the natural inclusion map $\SqFpInt \hookrightarrow \SqAfInt$.
We say that such a product over primes is an {\bf Euler product}, and refer to the sum for any given prime $\p$ as the {\bf Euler factor at $\p$}.

For any finite set of primes $\S$, we define the {\bf partial (non-archimedean) formal squareclass series} away from $\S$ by 
\begin{align}
\mathcal{F}^\S_{X, \bullet} 
& := 
	\[
	\prod_{\p \in\S} \frac{1}{(\Op^\times)^2}
	\]
	\cdot 
	\[
	\prod_{\p \notin\S}
	\sum_{S_\p \in \SqFpInt} \frac{X_{\p, \bullet}(S_\p)}{S_\p}
	\] 
\\
&= \sum_{
		\substack{S \in \SqAfInt \text{ where} \\ 
	 	\text{$S_\p = (\Op^\times)^2$ for all $\p \in \S$} 
		} 
 	} 
 	\frac{X_{\bullet}(S)}{S}.
\end{align}
For any subset $\mathbb{K} \subseteq \SqAfInt$ we also define the {\bf restriction to $\mathbb{K}$}
of a formal squareclass series by
$$
\mathcal{F}_{X, \bullet} \Big|_{\mathbb{K}} := \sum_{S\in \mathbb{K}} \frac{X_{\bullet}(S)}{S}.
$$

%%%%%%%%%%%%%%%%%%%%%%
%% Some Preliminary Local Definitions %%
%%%%%%%%%%%%%%%%%%%%%%
\medskip

In preparation for our main structure theorem (Theorem \ref{Theorem:M_N_Dirichlet}), we 
give some definitions useful for normalizing squareclasses and define some arithmetically interesting local quantities that will be used later for defining Euler factors.

\begin{defn}  \label{Defn:Normalized_squareclass}
Given $n \in \N$,
we say that a non-archimedean squareclass $\widetilde{S} \in \SqAfInt$ is {\bf normalized for $n$} if 
its associated ideal $\I(\widetilde{S}) = \mathfrak{h}_n$.  Similarly we say that a local squareclass $\widetilde{S_\p} \in \SqFpInt$ is {\bf normalized at $\p$ for $n$} if its associated ideal $\I(\widetilde{S_\p}) = \mathfrak{h}_n \Op$.
\end{defn}

\begin{defn}  \label{Defn:Uniformizing_squareclass}
We say $\pi_\p \in \SqFpInt$ is a {\bf  uniformizing squareclass (at $\p$)} 
if its associated ideal $\I(\pi_\p) = \p$.  A set $\P = \{\pi_\p\}$ consisting of one such $\pi_\p$ for each prime $\p$ of $F$ is called a {\bf family of uniformizing squareclasses}.  
\end{defn}

\begin{defn} \label{Defn:Associated_normalized_squareclass}
Given $n \in \N$ and a family of uniformizing squareclasses $\P = \{\pi_\p\}$,  
for any $S_\p \in \SqFpInt$ we can define the {\bf normalized squareclass $\widetilde{S_\p}$ associated to $S_\p$} by the formula $\widetilde{S_\p} := S_\p \cdot  \pi_\p^{\ord_\p(\mathfrak{h}_n / \I(S_\p))}$. 
Similarly, 
for any $S \in \SqAfInt$ we can define the {\bf normalized squareclass $\widetilde{S}$ associated to $S$} as $\widetilde{S} := S \cdot  \prod_{\p} \pi_\p^{\ord_\p(\mathfrak{h}_n / \I(S))}$, 
by using the natural inclusion $\SqFpInt \hookrightarrow \SqAfInt$. 
For $S \in \SqAfInt$ these two normalizations satisfy the local compatibility relation $(\widetilde{S})_\p = \widetilde{S_\p}$, and both  are normalized squareclasses in the sense of Definition \ref{Defn:Normalized_squareclass}.
\end{defn}

\begin{rem}
When defining the normalized squareclass $\widetilde{S}$ we often implicitly assume a fixed choice of a family  $\P$ of uniformizing squareclasses, and we will only refer to $\P$ explicitly when we are choosing a particular kind of family.
\end{rem}

\begin{defn}  \label{Def:generic_density_and_product}
Given $n \in \N$ and a normalized local squareclass $\widetilde{S_\p} \in \SqFpInt$ we define the  {\bf generic local density at $\p$} of rank $n$ to be 
$$
\beta_{n, \p}(\widetilde{S_\p}) := \beta_{G_\p, \p}(G_\p)
$$
where $G_\p$ is the unique local genus of primitive $\Op$-valued quadratic forms with $\det_H(G_\p) = \widetilde{S_\p}$, described in Lemma \ref{Lem:Hessian_det_ideals}.
%%%%
At the finitely many primes $\p\mid 2$ where such a genus may not exist for certain $S_\p$, we make an arbitrary (but forever fixed) choice of these generic local densities (e.g. say $\beta_{n, \p}(\widetilde{S_\p}) := 1$).
%%%%
Similarly, if $S \in \SqAfInt$ then we define the {\bf generic density product} of rank $n$ as the product 
$$
\beta_{n, \mathbf{f}}(\widetilde{S}) := \prod_\p \beta_{n, \p}(\widetilde{S}_\p).
$$
where $\widetilde{S}$ is the normalized squareclass associated to $S$.
\end{defn}

\begin{defn} \label{Def:normalized_mass_sums}
Given $n\in\N$ and  $\ve \in \{\pm1\}$, we define the {\bf normalized local mass sums} at $\p$ to be the quantities
$$
\widetilde{M}_{\p; n}^{*, \ve}(S_\p) :=
\sum_{
G_\p \in \mathbf{Gen}^*_{\p}(S, \ve; n)
%\substack{
%\text{Primitive local genera $G_\p$} \\
%\text{of $\Op$-integral quadratic forms} \\
%\text{in $n$-variables with} \\
%\text{$c_\p(G_\p) = \ve$ and $\det_H(G_\p) = S_\p$} \\
%%\text{ either $\det(G_p) \in \frac{1}{4} p^\nu \, (\Z_p^\times)^2$ if $p \neq 2$} \\
%%\text{ or $\det(G_p) \in \frac{1}{4} p^\nu u_2 \, (\Z_p^\times)^2$ if $p=2$} \\
%%%\text{ with $\det(G_p) \in \frac{1}{4} p^\nu \, (\Z_p^\times)^2$ } \\
%%\text{} \\
%}
}
\frac{\beta_{n, \p}(\widetilde{S_\p})}{\beta_{G_\p, \p}(G_\p)}.
$$
For convenience we also define the quantities %$A_{\p; n}(S_\p)$ and $B_{\p; n}(S_\p)$ by
$$
A^*_{\p; n}(S_\p) := \widetilde{M}_{\p; n}^{*, +}(S_\p) + \widetilde{M}_{\p; n}^{*, -}(S_\p)
\qquad
\text{and}
\qquad
B^*_{\p; n}(S_\p) := \widetilde{M}_{\p; n}^{*, +}(S_\p) - \widetilde{M}_{\p; n}^{*, -}(S_\p),
$$
which allow us to better understand the $\ve$-dependence of $\widetilde{M}_{\p; n}^{*, \ve}(S_\p)$ via the formula
\begin{equation} \label{Eq:M_as_A_and_B}
\widetilde{M}_{\p; n}^{*, \ve}(S_\p) = \frac{A^*_{\p; n}(S_\p) + \ve \cdot B^*_{\p; n}(S_\p)}{2}.
\end{equation}
\end{defn}

%%%%%%%%%%%%%%%%%%%%%%%%%%%%%%%%%
%%  Lemma:  The Euler Products of A and B are well-defined!  %%
%%%%%%%%%%%%%%%%%%%%%%%%%%%%%%%%%
\begin{lem} \label{Lem:A_and_B_generically_one}
If $\p\nmid 2$ and $S_\p \in \SqFpInt$,
%$G_\p$ is a local genus of non-degenerate $\Op$-valued quadratic forms with $S_\p := \det_H(G_\p)$, 
then
$$ 
A^*_{\p; n}(S_\p) = B^*_{\p; n}(S_\p) =
\begin{cases}
1 & \qquad\text{if $\ord_p(S_\p) = 0$,} \\
0 & \qquad\text{if $\ord_p(S_\p) < 0$,} \\
\end{cases}
$$
so the formal Euler products $\mathcal{F}_{A^*; n}$ and $\mathcal{F}_{B^*; n}$ respectively associated to the functions 
$A^*_n(S) := \prod_\p A^*_{\p; n}(S_\p)$ and $B^*_n(S) := \prod_\p B^*_{\p; n}(S_\p)$
make sense term-by-term as formal non-archimedean squareclass series.  
\end{lem}

\begin{proof}
By Lemma \ref{Lem:Hessian_det_ideals}, any local genus of non-degenerate $G_\p$ of $\Op$-valued quadratic forms
has associated ideal $\I(\det_H(G_\p)) \subseteq \mathfrak{h}_n \Op \subseteq \Op$, so
so there are no such local general with $\ord_\p(S_\p) < 0$ and so $A^*_{\p; n}(S_\p) = B^*_{\p; n}(S_\p) = 0$ in this case.

If $\p\nmid 2$ and $S_\p \in \SqFpInt$ is a local squareclass with $\ord_\p(S_\p) = 0$ then by Theorem \ref{Thm:Uniqueness_for_normalized_QF} there is a unique local genus $G_\p$ of quadratic forms with $\det(G_\p) = S_\p = \widetilde{S_\p}$.  Since $G_\p$ is unimodular and $\p\nmid 2$ we also have that  $c_\p(G_\p)= 1$, which shows that $A^*_{\p; n}(S_\p) = B^*_{\p; n}(S_\p) = 1$.
\end{proof}

%%%%%%%%%%%%%%%%%%%%%%%%%%%
%% Main Structure Theorem about the total mass  %%
%%%%%%%%%%%%%%%%%%%%%%%%%%%
%\medskip
\smallskip

We now give an explicit decomposition formula for $M^*_{\ve_\infty, \c_{\S}; n}(S)$ 
%$M_{N, \sigma}$ 
in the language of formal non-archimedean squareclass series.

\begin{thm} \label{Theorem:M_N_Dirichlet}
The formal non-archimedean squareclass series 
$$
\mathcal{F}_{M^*, \ve_\infty, \c_{\S}; n} 
:= \sum_{S \in \SqAfInt} 
	\frac{M^*_{\ve_\infty, \c_{\S}; n}(S)}{S}
$$
can be written as 
$$
\mathcal{F}_{M^*, \ve_\infty, \c_{\S}; n} 
=  
\hspace{-.2in}
\sum_{\substack{\text{normalized} \\ \widetilde{S} \in \SqAfInt}}
\hspace{-.2in}
\tfrac{1}{2} 
\beta_{n, \mathbf{f}}^{-1}(\widetilde{S})
	\Biggl(
	\mathcal{K}_{\c_\S; n} 
	\cdot 
		\[
		\mathcal{F}^\S_{A^*; n} + C_{\ve_\infty, \c_{\S}} \cdot \mathcal{F}^\S_{B^*; n}
		\]
	\Biggr|_{\widetilde{S}\cdot \langle\mathcal{P}\rangle}
$$
where 
$$
\mathcal{K}_{\c_\S; n} := 
%\[
\prod_{\p \in \S}
\,
\sum_{S_\p \in \SqFpInt}
\frac{\widetilde{M}_{\p; n}^{*, (\c_{\S})_\p}(S_\p)}{S_\p},
%\]
$$
$C_{\ve_\infty, \c_{\S}} := \ve_\infty %\prod_{v\mid\infty} c_v((\vec{\sigma}_\infty)_v) 
\cdot 
\prod_{\p\in\S} (\c_{\S})_\p \in \{\pm1\}$,
and 
$\langle\mathcal{P}\rangle$ is the multiplicative subset of $\SqAfInt$ generated by the (implicitly fixed) uniformizing family of squareclasses $\P$.
Here the series $\mathcal{F}^\S_{A^*; n}$ and $\mathcal{F}^\S_{B^*; n}$ are both given as Euler products over  primes $\p\notin\S$.
\end{thm}

\begin{proof}
Given $S \in \SqAfInt$, 
from Definition \ref{Def:M_idelic}
we have  
\begin{align}
\allowdisplaybreaks[4]
M^*_{\ve_\infty, \c_{\S}; n}(S)
&= 
\sum_{G_\mathbf{f} \in \mathbf{Gen}^*_{\mathbf{f}}(S, \ve_\infty, \c_{\S}; n)}  
	\prod_\p \beta_{G, \p}(G)^{-1}
=
\beta_{n, \mathbf{f}}^{-1}(\widetilde{S})
\cdot 
\sum_{G_\mathbf{f} \in \mathbf{Gen}^*_{\mathbf{f}}(S, \ve_\infty, \c_{\S}; n)}  
	\prod_\p 
	\frac{\beta_{n, \p}(\widetilde{S}_\p)}{\beta_{G, \p}(G)}
\\
& =
\beta_{n, \mathbf{f}}^{-1}(\widetilde{S})
\cdot 
\sum_{G_\mathbf{f} \in \mathbf{Gen}^*_{\mathbf{f}}(S, \ve_\infty, \c_{\S}; n)}
\prod_{\substack{
\p \in\T \text{ where} \\
\T := \{ \p\mid 2\} \cup \Supp(S) \cup \S 
}}
\frac{\beta_{n, \p}(\widetilde{S}_\p)}{\beta_{G, \p}(G)}
\\
& =
\beta_{n, \mathbf{f}}^{-1}(\widetilde{S})
\sum_{\substack{
(\ve_\p)_{\p\in\T} \in \{\pm1\}^{|\T|} \\
\text{satisfying} \\
\ve_\p =  (\c_{\S})_\p \text{ if $\p\in \S$ and } \\
\prod_{\p\in\T - \S} \ve_\p = C_{\ve_\infty, \c_{\S}} 
}}
\,
\sum_{\substack{
\text{Tuples of primitive local genera $G_\p$} \\
\text{of $\Op$-integral quadratic forms} \\
\text{in $n$-variables with $\p\in\T$} \\
\text{$c_\p(G_\p) = \ve_\p$ and $\det_H(G_\p) = S_\p$} \\
}}
\,\,
\prod_{\p \in \T} 
\,\,
\frac{\beta_{n, \p}(\widetilde{S}_\p)}{\beta_{G_\p, \p}(G_\p)}
\\
& =
\beta_{n, \mathbf{f}}^{-1}(\widetilde{S})
\sum_{\substack{
(\ve_\p)_{\p\in\T} \in \{\pm1\}^{|\T|} \\
\text{satisfying} \\
\ve_\p =  (\c_{\S})_\p \text{ if $\p\in \S$ and } \\
\prod_{\p\in\T - \S} \ve_\p = C_{\ve_\infty, \c_{\S}} 
}}
\,
\prod_{\p \in \T}
\,\,
\underbrace{
\sum_{
G_\p \in \mathbf{Gen}^*_{\p}(S, \ve; n)
}
\frac{\beta_{n, \p}(\widetilde{S}_\p)}{\beta_{G_\p, \p}(G_\p)}
}_{\widetilde{M}_{\p; n}^{*, \ve_\p}(S_\p) =
}.
\end{align}

By using (\ref{Eq:M_as_A_and_B})  
to re-express $\widetilde{M}_{\p; n}^{*, \ve_\p}(S_\p)$,
and Lemma \ref{Lemma:polynomial_id} below, we obtain
\begin{align}
\allowdisplaybreaks[4]
%M(S)
M^*_{\ve_\infty, \c_{\S}; n}(S)
& =
\beta_{n, \mathbf{f}}^{-1}(\widetilde{S})
\sum_{\substack{
(\ve_\p)_{\p\in\T} \in \{\pm1\}^{|\T|} \\
\text{satisfying} \\
\ve_\p =  (\c_{\S})_\p \text{ if $\p\in \S$ and } \\
\prod_{\p\in\T - \S} \ve_\p = C_{\ve_\infty, \c_{\S}} %c_{\infty}(\sigma) \prod_{\p\in\S} (\c_{\S})_\p
}}
\,
\prod_{\p \in \T}
\,
\frac{A^*_{\p; n}(S_\p) + \ve_\p B^*_{\p; n}(S_\p)}{2}
\\
& =
\beta_{n, \mathbf{f}}^{-1}(\widetilde{S})
\[
\prod_{\p \in \S}
\,
\widetilde{M}_{\p; n}^{*, (\c_{\S})_\p}(S_\p)
\]
\sum_{\substack{
(\ve_\p)_{\p\in\T-\S} \in \{\pm1\}^{|\T-\S|} \\
\text{satisfying} \\
\prod_{\p\in\T - \S} \ve_\p = C_{\ve_\infty, \c_{\S}} 
}}
\,
\prod_{\p \in \T-\S}
\,
\frac{A^*_{\p; n}(S_\p) + \ve_\p B^*_{\p; n}(S_\p)}{2}
\\
& =
\beta_{n, \mathbf{f}}^{-1}(\widetilde{S})
\[
\prod_{\p \in \S}
\,
\widetilde{M}_{\p; n}^{*, (\c_{\S})_\p}(S_\p)
\]
\,\,\,
\frac{2^{{|\T-\S|}-1}}{2^{|\T-\S|}}
\,
\[
\,
\prod_{\p \in \T-\S} A^*_{\p; n}(S_\p)
+ C_{\ve_\infty, \c_{\S}} 
\prod_{\p \in \T-\S} B^*_{\p; n}(S_\p)
\].
\end{align}

Since all of the valuations $\ord_\p(\det_H(Q)) \geq 0$ for any $\OF$-valued quadratic form $Q$,
we see that $A^*_{\p; n}(S_\p) = B^*_{\p; n}(S_\p) = 0$ when $\ord_\p(S) < 0$.  Also, if $\p\nmid 2$ and $\ord_\p(S)=0$ then there is a unique local genus of $\Op$-valued ternary quadratic forms of determinant $S_\p$, and this local genus has Hasse invariant $c_\p=1$, which shows that $A^*_{\p; n}(S_\p) = B^*_{\p; n}(S_\p) = 1$ when $\ord_\p(S) = 0$.

By Lemma \ref{Lem:Hessian_det_ideals} every $S \in \SqAfInt$ with $M^*_{\ve_\infty, \c_\S; n}(S) \neq 0$ has $\I(S) \subseteq \mathfrak{h}_n$ and becomes normalized by dividing by a product of powers of the uniformizing squareclasses $\pi_\p \in \mathcal{P}$, so 
$$
M^*_{\ve_\infty, \c_\S; n}(S) \neq 0 
\implies 
S \in \bigsqcup_{\substack{
		\text{normalized} \\ 
		\widetilde{S} \in \SqAfInt
		}} 
	\widetilde{S} \cdot \langle\mathcal{P}\rangle.
$$
%
%
%
%%%
This shows that the formal non-archimedean squareclass series 
$\mathcal{F}_{M^*, \ve_\infty, \c_{\S}; n}$
can be written as sum over squareclasses with the same normalization $\widetilde{S}$ 
of a linear combination of two formal squareclass series, each of which admits an Euler product expansion given by the desired formulas.
\end{proof}

We now prepare to associate a formal Dirichlet series to a formal non-archimedean squareclass series.

\begin{defn}[Distinguished squareclasses]  \label{Def:distinguished_squareclasses}
Given $n \in \N$, we say that a homomorphism $\lambda: I(\OF) \ra \SqAfInt$ defines a {\bf distinguished family of (non-archimedean) squareclasses} $\lambda(\mathfrak{a})$ if 
$\I(\lambda(\mathfrak{a})) = \mathfrak{a}$  for all $\mathfrak{a} \in I(\OF)$.
A family of distinguished squareclasses gives us a good way of parametrizing (a family of) Hessian determinant squareclasses of $\OF$-valued rank $n$ quadratic $\OF$-lattices by integral ideals.
\end{defn}

\begin{rem}[Normalizing distinguished squareclasses]  
\label{Rem:normalizing_distinguished_squarecalsses}
Because of Theorem \ref{Theorem:M_N_Dirichlet} it is important to understand how the 
normalization map
$\sim_n \,:S \mapsto \widetilde{S}$ affects our family of distinguished squareclasses $\lambda(\mathfrak{a})$.  We have the following commutative diagram
\begin{equation}  \label{Eq:Rational_reduction}
\begin{tikzpicture}
[baseline=(current bounding box.center)]
[description/.style={fill=white,inner sep=2pt}]
\matrix (m) [matrix of math nodes, row sep=3em,
column sep=2.5em, text height=1.5ex, text depth=0.25ex]
{ 
I(\OF) &  \SqAfInt \\
\SqCl(I(\OF)) & \SqAfInt \\
};
\path[->,font=\scriptsize]
%\path[right hook->] 
(m-1-1) edge node[auto] {$\lambda$} (m-1-2);
%\path[right hook->] 
\path[dashed,->,font=\scriptsize]
(m-2-1) edge node[auto] {$\widetilde{\lambda}$} (m-2-2);
%(m-1-1) edge node[auto] {$ \varphi $} (m-1-2)
%	edge node[description] {$ \Psi $} (m-2-2)
\path[->,font=\scriptsize] (m-1-2) edge node[auto] {$\sim_n$} (m-2-2);
\path[->>,font=\scriptsize] (m-1-1) edge node[auto] {id} (m-2-1);
\path[->,font=\scriptsize] (m-1-1) edge node[auto] {$\widetilde{\lambda}$} (m-2-2);
\end{tikzpicture}
\end{equation}
where the dashed map $\widetilde{\lambda}$ is well-defined because 
$$
\widetilde{\lambda(\mathfrak{a})} 
= \lambda(\mathfrak{a}) \prod_\p {\pi_\p}^{\ord_\p(\mathfrak{h}_n) - \ord_\p(\mathfrak{a})},
$$
so by (strict) multiplicativity we have that 
$$
\widetilde{\lambda(\mathfrak{a} \mathfrak{b}^2)}
= \widetilde{\lambda(\mathfrak{a})} 
	\cdot 
	\underbrace{
		\lambda(\mathfrak{b})^2 \prod_\p {\pi_\p}^{- 2\ord_\p(\mathfrak{b})}
	}_{= 1 \in \SqAfInt}
= \widetilde{\lambda(\mathfrak{a})}.
$$
The map $\widetilde{\lambda}$ is usually non-constant (see Remark \ref{Rem:canonical_distinguished_squareclasses} for the exception), and to account for this variation we write $I(\OF)$ as the disjoint union of squareclasses $\mathfrak{t} \cdot I(\OF)^2$
where $\mathfrak{t}$ varies over all squarefree ideals $\mathfrak{t} \in I(\OF)$.
\end{rem}

\begin{rem}[Canonical ``local'' distinguished squareclasses]  
\label{Rem:canonical_distinguished_squareclasses}
Given $n \in \N$, a normalized non-archimedean squareclass $\widetilde{S} \in \SqAfInt$, and a family $\P$ of uniformizing squareclasses, we have a canonical ``local'' distinguished family of squareclasses defined by the rule $\mathfrak{a} =: \prod_\p \p^{\nu_\p} \mapsto \lambda'(\mathfrak{a})$ where 
$$
\lambda'(\mathfrak{a}) := 
\lambda'_{\widetilde{S}, \P ;n}(\mathfrak{a}) := 
\begin{cases}
\prod_\p \pi_\p^{\nu_\p} \cdot \widetilde{S} %U_\mathbf{f}^2 
	& \qquad \text{if $n$ is even,} \\
\prod_\p \pi_\p^{\nu_\p - \ord_\p(2)} \cdot \widetilde{S} %U_\mathbf{f}^2 
	& \qquad \text{if $n$ is odd.}
\end{cases}
$$
However because this family is defined purely locally, it will not be as interesting as some of the more globally defined squareclasses that we will consider in Section \ref{Sec:Binary_Forms}.  These are the only families of distinguished squareclasses where the map $\widetilde{\lambda}$ is constant (having value $\widetilde{\lambda}(\mathfrak{a}) = \widetilde{\lambda}(\OF) = \widetilde{S}$).
\end{rem}

This notion of distinguished squareclasses will allow us to naturally associate formal Dirichlet series to a formal non-archimedean squareclass series.

\begin{defn}[Associated Dirichlet series] \label{Defn:Associated_Dirichlet_series}
Given 
a family $\lambda$ of distinguished  
squareclasses, we can associate to any formal non-archimedean squareclass series $\mathcal{F}_{X, \bullet; n}$ 
the formal Dirichlet series $D_{X, \lambda, \bullet ;n}$ by considering the terms associated to the distinguished squareclasses $\lambda(\mathfrak{a})$.
More generally, for any finite (possibly empty) set $\T$ of primes, we let 
$$
D^\T_{X, {\lambda}, \bullet; n}(S) :=  \sum_{\mathfrak{a}\in I^\T(\OF)} \frac{X_{\bullet; n}(\lambda(\mathfrak{a}))}{\mathfrak{a}^s}.
$$
\end{defn}

The following corollary shows that these formal Dirichlet series inherit much of the structure of the formal squareclass series they are derived from.

\begin{cor} \label{Cor:Formal_Dirichlet_M}
Given $n\in\N$, $\ve_\infty \in \{\pm1\}$, a vector of Hasse invariants $\c_\S$, and a family of distinguished squareclasses $\lambda$,
the formal Dirichlet series $D_{M, \lambda, \ve_\infty, \c_\S; n}(s)$ can be written as 
$$
D_{M^*, {\lambda}, \ve_\infty, \c_\S; n}(s) 
= 
\sum_{\mathfrak{t}} 
	\tfrac{1}{2} \beta_{n, \mathbf{f}}^{-1}(\widetilde{\lambda(\mathfrak{t})})
	\cdot 
%\Biggl(
\left(
	K_{\c_\S; n} 
	\cdot 
	%\sum_{\substack{Squarefree \\ \mathfrak{t} \in I^\S(\OF) }}
	\[ 
	D^\S_{A^*, \lambda; n}(s) + C_{\ve_\infty, \c_{\S}} \cdot D^\S_{B^*, {\lambda}; n}(s)
	\]
%\Biggr|
\right|_{\mathfrak{t} \, I(\OF)^2}
$$
where $\mathfrak t$ runs over all squarefree ideals in $I(\OF)$, 
$$
K_{\c_\S; n} := 
%\[
\prod_{\p \in \S}
\,
\sum_{i=0}^\infty
\frac{\widetilde{M}_{\p; n}^{*, (\c_{\S})_\p}(\lambda(\p^i))}{\p^i},
%\]
$$
and $C_{\ve_\infty, \c_{\S}} := \ve_\infty \cdot \prod_{\p\in\S} (\c_{\S})_\p \in \{\pm1\}$.
%is defined in Theorem \ref{Theorem:M_N_Dirichlet} 
%and 
Here 
both $D_{A^*, {\lambda}; n}^\S(s)$ and $D_{B^*, {\lambda}; n}^\S(s)$ are both given as Euler products over  primes $\p\notin\S$.
\end{cor}

\begin{proof}
The statement of the Corollary follows directly by taking the formal Dirichlet series of the statement in Theorem \ref{Theorem:M_N_Dirichlet} relative to the distinguished family of squareclasses $\lambda$, and by using Remark \ref{Rem:normalizing_distinguished_squarecalsses} to write the sum over normalized $\widetilde{S}$ as a sum over squareclasses of integral ideals.
The Euler product expansions follow because the maps $\mathfrak{a} \mapsto A^*_{\p;n}(\lambda(\mathfrak{a}))$ and $\mathfrak{a} \mapsto B^*_{\p;n}(\lambda(\mathfrak{a}))$ are multiplicative functions, and by Lemma \ref{Lem:A_and_B_generically_one}  the formal infinite products 
involve only finitely many non-trivial factors.
\end{proof}

%%%%%%%%%%%%%%%%%%%%%%%%%%%%%%
% Exact Formula for the generic product from GHY paper %%
%%%%%%%%%%%%%%%%%%%%%%%%%%%%%%
\begin{rem}[Computing the generic product]
We note that an exact formula for the generic product $\beta_{n, \mathbf{f}}^{-1}(\widetilde{S})$ can be derived from \cite[\textsection6, p115-120]{GHY} because the associated genera of local quadratic forms $Q_\p$ of determinant $\det_H(Q_\p) = \widetilde{S}_\p$  correspond to locally maximal $\Op$-valued quadratic lattices $L_\p$.  (However these $L_\p$ only assemble to a  global maximal lattice $L$ when $\widetilde{S}$ is globally rational.)  For each prime $\p$ these factors differ from the ``generic'' unimodular factors at $\p$ in $L(M)$  in \cite[Eq(7.2), p121]{GHY} by the factors $\lambda_\p$ given in \cite[Prop 4.4 and 4.5, p121]{GHY}.
When $n$ is odd this formula shows that $\beta_{n, \mathbf{f}}^{-1}(\widetilde{S})$ is constant as $\widetilde{S}$ varies, essentially because there is only one orthogonal group over each finite field.

\end{rem}

When $n$ is odd the previous remark considerably simplifies the statement of Corollary \ref{Cor:Formal_Dirichlet_M} by removing the dependence on the squarefree ideals $\mathfrak{t}$.

\begin{cor} \label{Cor:Formal_Dirichlet_decomp_n_odd}
Suppose that $n$ is odd.  Then the formal Dirichlet series $D_{M^*, \lambda, \ve_\infty, \c_\S; n}(s)$ 
of Corollary \ref{Cor:Formal_Dirichlet_M} can be written as 
$$
D_{M^*, {\lambda}, \ve_\infty, \c_\S; n}(s) 
= 
	\tfrac{1}{2} \beta_{n, \mathbf{f}}^{-1}(\widetilde{\lambda(\OF)})
	\cdot 
%\Biggl(
	K_{\c_\S; n} 
	\cdot 
	%\sum_{\substack{Squarefree \\ \mathfrak{t} \in I^\S(\OF) }}
	\[ 
	D^\S_{A^*, \lambda; n}(s) + C_{\ve_\infty, \c_{\S}} \cdot D^\S_{B^*, {\lambda}; n}(s)
	\],
%\Biggr|
$$
with the notation as defined there.
\end{cor}

\begin{rem}[Variation of the signature vector]
Notice that while there are  $(n+1)^{r}$ possible signature vectors $\vec{\sigma}_\infty$ for 
non-degenerate quadratic spaces over $F$ of given dimension $n$ (where $r$ is the number of real embeddings of $F$), there are only two possible series $\mathcal{F}_M$ and the dependence on $\vec{\sigma}_\infty$ is encoded in the single sign $\ve_\infty \in \{\pm1\}$.  This observation will allow us to translate questions about the masses of indefinite quadratic forms into question about totally definite forms, where explicit computations are more tractable.  This is done in later sections to numerically verify our local mass computations of $A_{\p; n}$ and $B_{\p; n}$.
\end{rem}

%%%%%%%%%%%%%%%%%
%%  Main Technical Lemma   %%
%%%%%%%%%%%%%%%%%
\smallskip
Finally we state and prove the main technical lemma of this section, used in the proof of Theorem \ref{Theorem:M_N_Dirichlet}.

\begin{lem} \label{Lemma:polynomial_id}
Suppose $\T$ is a nonempty finite set, $X_i$ and $Y_i$ are indeterminates for all $i\in\T$, and let $\mu_N$ denote the set of all $N^\text{th}$ roots of unity in $\C$.  Then for any $c\in\mu_N$ we have the polynomial identity
\begin{equation} \label{Eq:Polynomial_identity}
%\frac{1}{\varphi(N)^{|\S|-1}}
\sum_{\substack{
(\ve_i)_{i\in\T} \in {\mu_N}^{|\T|} \\
\text{satisfying} \\
\prod_{i\in\T} \ve_p = c}}
\,
\prod_{i\in\T} 
\,\,
(X_i + \ve_i Y_i)
= 
N^{|\T|-1}
\[
\prod_{i\in\T} X_i + c \prod_{i\in\T} Y_i 
\]
\, .
\end{equation}
\end{lem}

\begin{proof}
By dividing the identity by $\prod_{i\in\T} Y_i$ and replacing $X_i/Y_i$ by $X_i$, we can assume without loss of generality that all $Y_i = 1$.  For any finite set $\V$ we define a norm map $\Nm_{\V}:(\mu_N)^{\V} \twoheadrightarrow \mu_N$ on $\V$-tuples by $\Nm_{\V}(\x_{\V}) := \prod_{i \in \V} x_i$, and for finite sets  $\V' \subseteq \V$ we define restriction maps $\res^{\V}_{\V'}: (\mu_N)^{\V} \twoheadrightarrow (\mu_N)^{\V'}$ by 
$\res^{\V}_{\V'}(\x_{\V}) := (x_i)_{i\in \V'}$. 

Now consider the term $a_\U \prod_{i\notin\U} X_i$ on the left-hand side of (\ref{Eq:Polynomial_identity}) for some fixed $\U \subseteq \T$.  
%Here 
%$a_{\T}$ is the sum of all partial 
Then we have 
$$
a_{\U} = 
\sum_{\substack{
(\ve_i)_{i\in\T} \in {\mu_N}^{|\T|} \\
\text{satisfying} \\
\prod_{i\in\T} \ve_p = c}}
\,
\prod_{i\in\U} 
\,\,
\ve_i
 = 
 \sum_{\x \in \Nm_\T^{-1}(c)}
%\varphi_{\U}(\x)
\Nm_{\U}(\res^{\T}_{\U}(\x)).
$$
For convenience, we let $\varphi_{\U} := (\Nm_{\U} \circ \res^{\T}_{\U}|_{\Nm_\T^{-1}(c)}$.
If $\U = \T$ then  $\Image(\varphi_{\U}) = c$ and $\varphi_{\U}$ has multiplicity $N^{|\T| - 1}$, so $a_\T = c \cdot  N^{|\T| - 1}$.
If $\U = \emptyset $ then $\Image(\varphi_{\U}) = 1$ and $\varphi_{\U}$ has multiplicity $N^{|\T| - 1}$, so $a_\emptyset = N^{|\T| - 1}$.
If $\emptyset \subsetneq \U \subsetneq \T$ then $\Image(\varphi_{\U}) = \mu_N$ and each fibre $\varphi_{\U}^{-1}(\ve)$ has multiplicity $N^{|\T| - 2}$, so $a_\T = N^{|\T| - 2} \sum_{\ve \in \mu_N} \ve = 0$.
\end{proof}

%%%%%%%%%%%%%%%%%%%%%%%
%%  Removing the primitivity assumption  %%
%%%%%%%%%%%%%%%%%%%%%%%
\begin{rem}[Removing Primitivity]
For some applications it is much more natural to remove the condition that we consider only primitive quadratic lattices $(L, Q)$ in the primitive total non-archimedean mass $M^*_{\vec{\sigma}_\infty; n}(S)$, say where $\S = \emptyset$ for simplicity.  This has the effect of multiplying the formal 
associated 
Dirichlet series $D_{M^*, \lambda, \ve_\infty; n}(s)$ by the factor by $\zeta_F(ns + \frac{n(n-1)}{2})$ for every choice $\lambda$ of a distinguished family of squareclasses.  This follows from the scale invariance of the mass $\Mass(L,Q)$, with the $ns$ term coming from fact that scaling a quadratic space by $c$ scales its determinant by $c^n$, and the $\frac{n(n-1)}{2}$ term arising from the variation of the local densities at $v\mid\infty$.
\end{rem}

%%%%%%%%%%%%%%%%%%%%%
%%  Connection with modular forms  %%
%%%%%%%%%%%%%%%%%%%%%
\begin{rem}[Relation to modular forms]
There are some structural similarities between our Dirichlet series in Corollary \ref{Cor:Formal_Dirichlet_M} and the (naive) Dirichlet series associated to a Hecke eigenform $f(z)$ of half-integral weight via the Mellin transform.  In particular, as pointed out by Shimura in \cite{Shimura:1973kx}, the action of the half-integral weight Hecke algebra on a Hecke eigen-cuspform $f(z) = \sum_{m\geq 1} a_m e^{2\pi i m z}$ produces relations between Fourier coefficients $a_m$  for $m$ within any fixed squareclass $t\N^2$ (with $t$ squarefree), and offers no insight about the squarefree coefficients $a_t$, which must be investigated separately.  This similarity might lead one to hope for the existence of a modular form whose $L$-function is essentially the Dirichlet series $D_{M^*, \lambda, \ve_\infty; n}(s)$ or $D_{M, \lambda, \ve_\infty; n}(s)$ of Definition \ref{Defn:Associated_Dirichlet_series}.

In \cite[Thrm2, p91]{Hirzebruch:1976fk} Hirzebruch and Zagier, following ideas of H. Cohen \cite{Cohen:1975uq}  in higher weights, show that there is a non-holomorphic function $\mathcal{H}(z)$ that transforms as a weight $3/2$ modular form for $\Gamma_0(4)$, whose holomorphic part is a Fourier series generating function for the positive definite total masses $T_{n=2}(m)$ (written there as the Hurwitz numbers $H(m)$) when $m \in \N$.  Furthermore,  the function $\mathcal{H}(z)$ naturally arises as a linear combination of Eisenstein series from the two regular cusps.  Here the total mass Dirichlet series $D_{T; n=2}(s)$ over $\Q$ agrees with the usual $L$-function $L(\mathcal{H}, s)$ of $\mathcal{H}(z)$.
Similarly, when $n=1$ the total mass Dirichlet series $D_{T; n=1}(s)$ for positive definite forms over $\Q$ agrees with the Riemann zeta function $\zeta(s)$, which can be thought of the $L$-function of an Eisenstein series on $\GL_1$.  

These results and structural similarities suggest that the total mass Dirichlet series $D_{M, \lambda, \ve_\infty; n}(s)$ arise as $L$-functions of Eisenstein series on $\GL_n$, and that when $n$ is even the weight should be half-integral (i.e. these are automorphic forms on a double cover of $\GL_n$).  This phenomenon will be investigated further in future papers, and as a first step in this direction, we give explicit formulas for the ternary case $D_{T; n=3}(s)$ in \cite{Hanke_n_equals_3_masses}.
\end{rem}

%%%%%%%%%%%%%%%%%%%%%%%%%%%%%%%
%% -----------------------------------------------------------------------------
%%  End of Main definitions of the Local Density Formal series
%% -----------------------------------------------------------------------------
%%%%%%%%%%%%%%%%%%%%%%%%%%%%%%%

%%%%%%%%%%%%%%%%%%%%%%%%%%%%%%%%%%%%%%%%%%%%%%%%%%%
%%%%%%%%%%%%%%%%%%%%%%%%%%%%%%%%%%%%%%%%%%%%%%%%%%%
%%%%%%%%%%%%%%%%%%%%%%%%%%%%%%%%%%%%%%%%%%%%%%%%%%%

%%%%%%%%%%%%%%%%%%%%%%%%%%%%%%
%% =========================================
%%  Structural Results for Euler factors 
%% =========================================
%%%%%%%%%%%%%%%%%%%%%%%%%%%%%%

\section{Structural results for the Euler factors $A^*_{\p; n}$ and $B^*_{\p; n}$}

In this section we establish the rationality of the Euler factors for $A^*_{\p; n}(S_\p)$ and $B^*_{\p; n}(S_\p)$ as $S_\p$ varies but its normalized squareclass $\widetilde{S_\p}$ is fixed
 (under some mild removable conditions when $\p\mid 2$).
 We also use local scaling symmetries when $n$ is odd to gain precise information about how the variation of $\widetilde{S_\p}$ affects these Euler factors (at all primes $\p$).

\begin{defn} \label{Defn:Jordan_structure_and_PLGS}
If $\p\nmid 2$ we define a {\bf Jordan block structure} of size $n \in \N$ as a tuple $(n_1, \cdots, n_r)$ of $n_i \in \N$ where $n_1 + \cdots + n_r = n$, but $r \in \N$ is not specified.
If $\p\mid 2$ and $F_\p = \Q_2$ then we define a 
%{\bf Jordan block structure} of size $n$, or 
{\bf partial local genus symbol} of size $n$, as Jordan block structure where 
\begin{enumerate}
\item the commas separating the elements of the tuple can be either a comma `,', a semicolon `;', or a pair of colons `::', 
\item every $n_i \in 2\N$ may appear either with a bar above it or not,
\item the semicolon separator may only appear between two (adjacent) unbarred numbers.
\end{enumerate}
  For example, the symbol $(1; 2, 1, \bar2, 1 :: 3)$ partial local genus symbol of size 10.
\end{defn}

\begin{rem}[Relation to genus symbols] \label{Rem:CS_PLGS_translation}
The partial local genus symbols are in bijective correspondence with the local genus symbols in \cite[Ch 10.7, pp378-384]{CS} for the prime $p=2$, where the oddities and signs are not specified.  In the Conway-Sloane notation, the `::' symbols indicate a separation between trains, and the barred numbers $\bar{n_i}$ indicate Type II Jordan blocks of dimension $n_i$ (which separate compartments iff they are between odd numbers in the same train), and the `;' indicates a separation of compartments (but not trains) between Type I blocks, indicating that their scale ideals differ by a factor of 4.

When $p \neq 2$, the Jordan block structures are in correspondence with the local genus symbols where the signs are not specified.  The invariants described over $\Q$ there are known to hold for any $\p\nmid 2$ for any number field $F$.
\end{rem}

\begin{rem}[Counting Jordan Block Structures] 
One can show that there are exactly $2^{n-1}$ Jordan block structures of size $n$ (for $\p\nmid 2$), because they satisfy the recursion that $J(n) = \sum_{i=0}^{n-1} J(i)$ where $J(0)=J(1) = 1$ and $J(n)$ is the number of Jordan block structures of size $n$.  Counting partial local genus symbols (for $\p\mid 2$ where $F_\p= \Q_2$) can also be done, but it is a much more complicated task.
\end{rem}

\begin{thm}[Rationality of ``fixed unit'' local series] \label{Thm:Rationality_of_Euler_factors}
If  $\p\nmid 2$ or $F_\p = \Q_2$, then %the functions $S_\p
the formal local Dirichlet series 
$$
\sum_{
	\substack{S_\p \in \SqFpInt \\ \text{with $\widetilde{S_\p}$ fixed}}
	}
%	\hspace{-.2in}
	\frac{A^*_{\p; n}(S_\p)}{N_{F/\Q}(\I(S_\p))^{s}}
\quad
\text{ and }
\quad 
\sum_{
	\substack{S_\p \in \SqFpInt \\ \text{with $\widetilde{S_\p}$ fixed}}
	}
%	\hspace{-.2in}
%	B_{\p; n}(S_\p) N_{F/\Q}(\I(\widetilde{S_\p}))^{-s}
	\frac{B^*_{\p; n}(S_\p)}{N_{F/\Q}(\I(S_\p))^{s}}
$$
%$A_{\p; n}(S_\p)$ and $B_{\p; n}(S_\p)$
are rational functions of $X := q^{-s}$, where the norm $N_{F/\Q}(\mathfrak{a}) := |\OF/\mathfrak{a}\OF|$.
%as we vary over all squareclasses $S_\p$ with fixed normalized squareclass $\widetilde{S}_\p$.
\end{thm}

\begin{proof}

First assume that $\p\nmid 2$ and 
consider a given Jordan block structure $J$ of size $n$ (i.e. dimensions and scales of the modular lattice summands) arising from $n$-variable primitive $\Op$-valued quadratic forms over $F_\p$, say with $r$ non-zero blocks.  From \cite[\textsection92:2, p147]{OM} we see that by decorating each (non-zero) Jordan block by a valuation zero squareclasses of $(F_\p)^\times$, we parametrize all $\Op$-equivalence classes of quadratic forms with that Jordan block structure $J$.  This decoration process gives rise to a distribution of Hasse invariants and normalized local densities for each Jordan block structure that gives a contribution to $A^*_{\p; n}(S_\p)$ and $B^*_{\p; n}(S_\p)$ (which are the sum of all such contributions).  We observe that this contribution $c_J$ varies with the scales of its components as 
%a series of the form
$$
c_J = 
\widetilde{c}_J \cdot
\prod_{i} q^{\kappa_i \al_i}
\qquad \qquad 
\text{where 
$\kappa_i := \sum_{j < i} \tfrac{n_i n_j}{2} - \sum_{j > i} \tfrac{n_i n_j}{2}$,
}
$$
$n_i := \dim(L_i)$, $\p^{\al_i} := s(L_i)$ and 
$\widetilde{c}_J \in \Q $ is independent of the scale valuations $\al_i$.  More precisely, $\widetilde{c}_J$ depends only on the tuple $\vec{n}_J := (n_i)$ of Jordan block dimensions $n_i \in \N$ with $\sum_i n_i = n$ (and in any associated Jordan block structure $J$ we have $\s(L_i) \supsetneq \s(L_j)$ when $i < j$).  Since there are only only finitely many such tuples $\vec{n}_J$ for any given $n\in \N$,  we have a finite sum of contributions $\widetilde{c}_J$ (indexed by $\vec{n}_J$) which vary with $\al := \ord_\p(S_\p) = \sum_i \al_i$ as sums of the form
$$
c_J = 
\widetilde{c}_J \cdot
\sum_{
\substack{
	0=\al_1 < \al_2 < \cdots < \al_r \\
	\text{ where } \sum_i \al_i = \al
}
}
\prod_{i=1}^r q^{\kappa_i \al_i}.
$$
By using Ferrer diagrams, we can use an affine change variables to rewrite these sums as free sums over $\vec{\beta} := (\beta_1, \cdots, \beta_{r})$ with $\beta_i \geq 0$ where each $\al_i$ is an affine function of $\vec{\beta}$.  Therefore the contribution of each of these terms to the formal Euler factors are products of geometric series, and a finite sum of geometric series is a rational function.

When $\p\mid 2$ and $F_\p = \Q_2$ then the same argument applies where we replace the Jordan block structure by partial local genus symbols $J$ on size $n$, by using the $2$-adic integral local invariants in \cite[\textsection7.3-6, pp380-3]{CS-book}.  For each partial local genus symbol, using \cite[\textsection4]{CS} we notice that the variation of $c_J$ with the scales of the trains comes purely from the ``cross product'' factor, and the sum over all such (finitely many) terms formally gives a sum of geometric series.
\end{proof}

The rationality result for ``fixed unit'' local series in Theorem \ref{Thm:Rationality_of_Euler_factors} allows us to show that the Euler factors of the Dirichlet series $D_{A^*, \lambda, \ve_\infty; n}(s)$ and $D_{B^*, \lambda, \ve_\infty; n}(s)$ are also rational.

\begin{cor}[Rationality of Euler factors] \label{Cor:Rationality_of_AB_Euler_factors}
If  $\p\nmid 2$ or $F_\p = \Q_2$, $\ve \in \{\pm1\}$ and $\lambda$ is a distinguished family of squareclasses, then %the functions $S_\p
the formal local Dirichlet series 
$$
\sum_{i=0}^\infty
%	\hspace{-.2in}
	\frac{A^*_{\p; n}(\lambda(\p^i))}{N_{F/\Q}(\p^i)^{s}}
\quad
\text{ and }
\quad 
\sum_{i=0}^\infty
%	\hspace{-.2in}
	\frac{B^*_{\p; n}(\lambda(\p^i))}{N_{F/\Q}(\p^i)^{s}}
$$
are rational functions of $X := q^{-s} = N_{F/\Q}(\p)^{-s}$.
These are the Euler factors at $\p$ of the Dirichlet series $D_{A^*, \lambda, \ve_\infty; n}(s)$ and $D_{B^*, \lambda, \ve_\infty; n}(s)$ described in 
\end{cor}

\begin{proof}
These local Dirichlet series agree with the local series described in Theorem \ref{Thm:Rationality_of_Euler_factors}, up to an overall scaling and possibly an alternating sign.  Since any scaling of a power series $F(x)= \sum_{i} a_i x^i$ along congruence classes of powers of $i$ can be accomplished by a rational transformation of $F(x)$, the result follows.
\end{proof}

\begin{rem}[Rationality for dyadic primes]
The condition that $F_\p = \Q_2$ for $\p\mid 2$ in Theorem \ref{Thm:Rationality_of_Euler_factors} is only present because the invariant theory of integral quadratic forms over general dyadic field \cite{OMeara:1955rr, OMeara:1957cr} has not been formulated in a way that makes it readily applicable here.  However one can formulate it in a language of partial local genus symbols (``trains and compartments'') that are decorated by rational invariants of the underlying quadratic spaces (under some equivalences which specialize to ``oddity fusion'' and ``sign walking'' over $\Q_2$) that allow the proof above to work for all $\p\mid 2$.  This reformulation will appear in a forthcoming paper \cite{Hanke:dyadic}.
\end{rem}

We now examine the effect of local unit scalings when $n$ is odd, which allow us to establish symmetries of the Euler factors for $A^*_{\p; n}$ and $B^*_{\p; n}$ across normalized local squareclasses.

\begin{thm}[Normalized local squareclass invariance]
Suppose that $n \in \N$ is odd and that $u_\p \in \Op^\times$.  Then 
$$
A^*_{\p; n}(u_\p S_\p) = A^*_{\p; n}(S_\p)
$$ 
and
$$
B^*_{\p; n}(u_\p S_\p) = 
\begin{cases}
B^*_{\p; n}(S_\p)  & \text{if $n \equiv 1 (\text{mod } 4)$,} \\
(u_\p, u_\p)_\p \cdot B^*_{\p; n}(S_\p)  & \text{if $n \equiv 3 (\text{mod } 4)$.} \\
\end{cases}
$$ 
\end{thm}

\begin{proof}
Consider the effect of the local scaling $G_\p \mapsto u_\p \cdot G_\p$ on the local genera $G_\p$ of $\det_H(G_\p) = S_\p$.  The scaled determinant is given by $\det_H(u_\p G_\p) = u_\p^{n} S_\p = u_\p S_\p$, and by Lemma \ref{Lem:Hilbert_symbol_scaling} we see that the scaled Hasse invariant is given by 
\begin{align*}
c_\p(u_\p S_\p) 
& =  
\textstyle{
	(u_\p, u_\p)_\p^\frac{n(n-1)}{2} 
	\cdot (u, 2^n  \det_H(G_\p))_\p^{n-1} 
	\cdot c_\p(S_\p) 
} 
\\
& = 
c_\p(S_\p) \cdot
\begin{cases}
1  & \text{if $n \equiv 1 (\text{mod } 4)$,} \\
(u_\p, u_\p)_\p  & \text{if $n \equiv 3 (\text{mod } 4)$.} \\
\end{cases}
\end{align*}
These formulas show the desired symmetries because the local unit scaling is an involution on local genera $G_\p$ that gives a bijection between genera of determinants $S_\p$ and $u_\p S_\p$, and also preserves their local densities (by the definition of $\beta_{Q, \p}(Q)$).
\end{proof}

\begin{rem}
If $\p \nmid 2$ and $n$ is odd then we see that the Euler factors $A^*_{\p; n}(S_\p)$ and $B^*_{\p; n}(S_\p)$ depend only on $\ord_\p(S_\p)$, or equivalently, on the local ideal associated to $S_\p$. When $n \equiv 1\pmod4$ this invariance holds at all primes $\p$.
\end{rem}

\begin{rem}[Symmetries for $n$ even]
When $n$ is even it is still possible to gain some symmetry information about the series $\mathcal{F}_{M^*, \ve_\infty, \c_{\S}; n}$ by examining the effect of global unit scalings by $u \in \OF^\times / (\OF^\times)^2$, but this only gives finitely many symmetries.
\end{rem}

\begin{lem} \label{Lem:Hilbert_symbol_scaling}
Given non-degenerate $n$-dimensional quadratic space $(V, Q)$ over a local field $K_v$ of characteristic $\neq 2$, we have
$$
\textstyle{
c_v(u\cdot V) = (u,u)^\frac{n(n-1)}{2}_v \cdot (u, \det_G(V))_v^{n-1} \cdot c_v(V).
}
$$
\end{lem}

\begin{proof}
Since $K_v$ has characteristic $\neq 2$, we can assume that $Q \sim_{K_v} a_1x_1^2 + \cdots + a_n x_n^2$ with $a_i \in K_v^\times$, giving $c_v(V) = \prod_{1 \leq i < j \leq n} (a_i, a_j)_v$.  From basic properties of the Hilbert symbol we have that 
$$
(u a_i, u a_j)_v = (u,u)_v \cdot (u, a_i a_j)_v \cdot (a_i, a_j)_v,
$$
which gives the desired formula by noticing that the factors $(u, a_i a_j)_v$ for a given index $i$ (with varying index $j$) appear exactly $n-1$ times.
\end{proof}

%%%%%%%%%%%%%%%%%%%%%%%%%%%%%%%
%% -----------------------------------------------------------------------------
%%  End of Structural Results for Euler factors
%% -----------------------------------------------------------------------------
%%%%%%%%%%%%%%%%%%%%%%%%%%%%%%%

%%%%%%%%%%%%%%%%%%%%%%%%%%%%%%%%%%%%%%%%%%%%%%%%%%%
%%%%%%%%%%%%%%%%%%%%%%%%%%%%%%%%%%%%%%%%%%%%%%%%%%%
%%%%%%%%%%%%%%%%%%%%%%%%%%%%%%%%%%%%%%%%%%%%%%%%%%%

%%%%%%%%%%%%%%%%%%%%%%%%%%%%%%
%% =========================================
%%  Relationship with the Mass formula
%% =========================================
%%%%%%%%%%%%%%%%%%%%%%%%%%%%%%

\section{Relationship with the Mass formula}

In this section we relate the primitive total adelic non-archimedean mass $M^*_{\ve_\infty, \c_{\S}; n}(S)$ with the total mass of quadratic lattices (as a sum over global classes) of any given signature vector $\vec{\sigma}_\infty$ and Hessian determinant squareclass $S$, where we may restrict the allowed Hasse invariants $c_\p$ at finitely many primes.

\begin{defn}
We define the {\bf proper mass} $\Mass^{+}(L)$ of an $\OF$-valued non-degenerate quadratic lattice $L$ to be the quantity $\Mass(\Lambda, \varphi)$  defined in \cite[eq (5.4), p26]{Ha-Thesis}.  Since $\Mass^{+}(L)$ only depends on the genus $G := \Gen(L)$, we also define the proper mass of a genus $G$ as $\Mass^{+}(G) := \Mass^{+}(L)$ for any $L\in G$.
\end{defn}

\begin{thm} \label{Thm:Mass_theorem}
%Suppose we are 
For a given a signature vector $\vec{\sigma}_\infty$ of rank $n$, Hasse vector $\c_{\S}$ for $F$, 
and 
globally rational non-archimedean squareclass $S \in \SqCl(\A_{F, \mathbf{f}}^\times, U_\mathbf{f})$,
%$S \in \SqCl(F^\times, \OF^\times)$ 
we have 
%.  Then 
\begin{align}
M^*_{\ve_\infty, \c_{\S}; n}(S)  
&= \frac{C_{\vec{\sigma}_\infty, n}}{2 \cdot |\Delta_F|^\frac{n(n-1)}{4}} 
\cdot
\frac{N_{F/\Q}(2\OF)^\frac{n(n+1)}{2}}{N_{F/\Q}(\I(S))^\frac{n+1}{2}}
\cdot
\sum_{
G \in \mathbf{Gen}^*(S, \vec{\sigma}_\infty, \c_{\S}; n)
}
\mathrm{Mass}^{+}(G)
\end{align}
where 
$$
C_{\vec{\sigma}_\infty, n} := 
\[
\prod_{\text{$v$ real}} 
2 \cdot 2^\frac{-n \cdot \min\{\sigma_{v,+}, \sigma_{v,-}\} }{2} \cdot V(\sigma_{v,+}) \cdot V(\sigma_{v,-})
\]
\[
\prod_{\text{$v$ complex}} 
2^\frac{n(n-3)}{2} \cdot V(n)
\],
$$
$\ve_\infty := \prod_{v\mid\infty} c_v((\vec{\sigma}_\infty)_v)$, 
and $V(r) := \frac{1}{2} \pi^\frac{r(r+1)}{4} \cdot \(\prod_{i=1}^r \Gamma(\frac{i}{2})\)^{-1}$ for $r \in \Z \geq 0$.
\end{thm}

\begin{proof}

By \cite[eq (5.7), p27]{Ha-Thesis} (which uses the convention that $\sigma_{v,+} \geq \sigma_{v,-}$ for every real place) we see that 
$$
\mathrm{Mass}^{+}(L) 
= 2 \cdot |\Delta_F|^\frac{n(n-1)}{4} \cdot
%\prod_{v\mid\infty} \(\frac{\cancel{|\det_G(\phi_v)|_v}}{|\det_G(Q)|_v}\)^{\frac{n+1}{2}}
\prod_{v\mid\infty} \Vol_{Q_v}(C_{v})^{-1}
\cdot 
\prod_\p \beta_{\p}(L, Q)^{-1},
$$
where $\Vol_{Q_v}(C_{v})$ is the volume on the fibre $C_v \subseteq SO(\phi_v)$ defined in \cite[\textsection4]{Ha-Thesis}, computed using the volume form on $SO(Q_v)$.  Since $\prod_{v\mid\infty}\Vol_{\phi_v}(C_{v}) = C_{\vec{\sigma}_\infty, n}$, we can use \cite[Lemma 2.2]{Ha-Thesis} to write 
$$
\prod_{v\mid\infty}\Vol_{Q_v}(C_{v})^{-1}
= \prod_{v\mid\infty} 
	\(\frac{|\det_G(Q)|_v}{\cancel{|\det_G(\phi_v)|_v}} \)^{\frac{n+1}{2}}
	C_{\vec{\sigma}_\infty, n}^{-1}.
$$
We can similarly use \cite[Lemma 2.2]{Ha-Thesis} to express the lattice local densities $\beta_{\p}(L, Q)$ in terms of the local densities of the quadratic form $\psi_\p$ induced by restricting $Q$ to $L_\p$ in some local basis of $L_\p$, giving 
$$
\prod_\p \beta_{\p}(L, Q)^{-1}
= \prod_\p \beta_{\psi_\p, \p}(\psi_\p)^{-1}
	\cdot \prod_\p 
		\(\sqrt{\frac{|\det_G(\psi_\p)|_\p}{|\det_G(Q)|_\p}} \)^{-(n+1)}.
$$
Now
%Putting these together and 
using the product formula and observing that $\prod_\p |\det_G(\psi_\p)|_\p^{-1} = N_{F/\Q}( \frac{1}{2^n}\I(S))$, gives 
$$
\mathrm{Mass}^{+}(L) 
= \frac{2 \cdot |\Delta_F|^\frac{n(n-1)}{4}}{C_{\vec{\sigma}_\infty, n}}
\cdot 
\frac{N_{F/\Q}(\I(S))^\frac{n+1}{2}}{N_{F/\Q}(2\OF)^\frac{n(n+1)}{2}}
\cdot 
\prod_\p \beta_{\psi_\p, \p}(\psi_\p)^{-1}.
$$
Finally, summing over all primitive quadratic rank $n$ lattices with the given signature vector and Hasse invariants at $\p \in \S$ gives $T^*_{\vec{\sigma}_\infty, \c_{\S}; n}(S)$, and using Lemma \ref{Lem:T_is_M} we recover $M^*_{\ve_\infty, \c_{\S}; n}(S)$.
\end{proof}

\begin{cor} \label{Cor:Mass_formula_result_for_globally_rational}
If the signature vector $\vec{\sigma}_\infty$ is totally definite and 
$S \in \SqCl(\A_{F, \mathbf{f}}^\times, U_\mathbf{f})$ is globally rational, 
%$S \in \SqCl(F^\times, \OF^\times)$, 
then
$$
M^*_{\ve_\infty, \c_{\S}; n}(S)  
= \frac{V(n)^{[F:\Q]}}{ |\Delta_F|^\frac{n(n-1)}{4}} 
\cdot
\frac{N_{F/\Q}(2\OF)^\frac{n(n+1)}{2}}{N_{F/\Q}(\I(S))^\frac{n+1}{2}}
\cdot
\sum_{
L \in \mathbf{Cls}^*(S, \vec{\sigma}_\infty, \c_{\S}; n)
%\substack{
%	\text{ Primitive genera $G$ of rank $n$ } \\
%	\text{ quadratic lattices $L$ of } \\
%	\text{ signature vector $\vec{\sigma}_\infty$} \\
%	\text{ where $\det_H(L) = S$ and } \\
%	\text{ $c_\p(L) = (\c_{\s})_\p$ for all $\p\in\S$ } \\
%%	\text{  } \\
%}
}
\frac{1}{|\Aut(L)|}
$$
where $\ve_\infty$ and $V(n)$ are defined in Theorem \ref{Thm:Mass_theorem}.
\end{cor}

\begin{proof}
This follows from Theorem \ref{Thm:Mass_theorem} by taking all signatures $\sigma_v := (n, 0)$, and noticing that $C_{\vec{\sigma}_\infty, n} = \prod_{v\mid\infty} V(n) = V(n)^{[F:\Q]}$.  We can replace the proper mass by the mass because $\Mass^+(G) = 2\Mass(G)$, which can be proved by analyzing how classes split into proper classes.
\end{proof}

\begin{cor}
Suppose that $\lambda$ gives a family of distinguished squareclasses where each $\lambda(\mathfrak{a})$ is globally rational 
and let $\vec{\sigma}_\infty$ be a totally definite signature vector, 
then the formal Dirichlet series
$$
\sum_{\mathfrak{a} \in I(\OF)} 
\( 
\sum_{
Q \in \mathbf{Cls}^*(\lambda(\mathfrak{a}), \vec{\sigma}_\infty, \c_{\S}; n)
}
\frac{1}{|\Aut(Q)|}
\)
\mathfrak{a}^{-s}
=
\frac{|\Delta_F|^\frac{n(n-1)}{4}}{\( 2^\frac{n(n+1)}{2} \cdot V(n) \)^{[F:\Q]}} 
\cdot
D_{M^*, \lambda, \ve_\infty, \c_\S; n}(s - \tfrac{n+1}{2})
$$
%where $S_n(\mathfrak{a})$ is relative to $\widetilde{S}$, and 
where $\ve_\infty$, $\c_\S$ and $V(n)$ are defined in Theorem \ref{Thm:Mass_theorem}.
The Dirichlet series $D_{M^*, \lambda, \ve_\infty, \c_\S; n}(s)$ is given explicitly in Corollary \ref{Cor:Formal_Dirichlet_M} as a linear combination of two Dirichlet series each of which is defined by an Euler product.
\end{cor}

%\begin{cor}
%Work out the factor in the totally positive definite case!
%\end{cor}

\begin{proof}
This follows directly from Corollary \ref{Cor:Mass_formula_result_for_globally_rational}. 
\end{proof}

%%%%%%%%%%%%%%%%%%%%%%%%%%%%%%%
%% -----------------------------------------------------------------------------
%%  End of Relationship with the Mass formula
%% -----------------------------------------------------------------------------
%%%%%%%%%%%%%%%%%%%%%%%%%%%%%%%

%%%%%%%%%%%%%%%%%%%%%%%%%%%%%%%%%%%%%%%%%%%%%%%%%%%
%%%%%%%%%%%%%%%%%%%%%%%%%%%%%%%%%%%%%%%%%%%%%%%%%%%
%%%%%%%%%%%%%%%%%%%%%%%%%%%%%%%%%%%%%%%%%%%%%%%%%%%

%%%%%%%%%%%%%%%%%%%%%%%%%%%%%%%%%%%%%%%
%% ========================================================
%%  Results for binary quadratic forms
%% ========================================================
%%%%%%%%%%%%%%%%%%%%%%%%%%%%%%%%%%%%%%%

\section{Results for binary quadratic forms} \label{Sec:Binary_Forms}

In this section we work out the formal non-archimedean squareclass series for binary quadratic forms, which is closely related to class numbers of quadratic rings and modular forms of weight $\frac{3}{2}$.

\begin{lem} \label{Lem:Count_binary_Jordan_and_PLGS}
When $n=2$ there are exactly two Jordan block structures and five partial local genus symbols.  Explicitly, they are $\{(2), (1,1)\}$ and $\{(\bar{2}), (2), (1,1), (1;1), (1::1) \}$.
\end{lem}

\begin{proof}
This follows from Definition \ref{Defn:Jordan_structure_and_PLGS}
by counting and decorating the ordered partitions of $n=2$.
\end{proof}

\begin{lem}[Distributions of Hasse invariants]  
Suppose that $G_\p$ varies over all primitive local genera of $\Op$-valued quadratic forms in 2 variables where $\det_H(G_\p)$ is fixed.  The distribution of Hasse invariants $c_\p$ arising from genera $G_\p$ above with fixed Jordan block structures (for $\p\nmid 2$) is given by
$$
\begin{tabular}{c | c | c | c | c}
\text{Jordan block structure} & $\det_H(G_\p)$ & $c_\p =1$ & $c_\p =-1$ & Conditions\\
\hline
$(2)$ & -- & $1$ & $0$ & --\\
\hline
\multirow{2}{*}{$(1,1)$} & \multirow{2}{*}{$u \pi_\p^\nu$} & $1$ & $1$ & if $\nu$ is odd \\
\cline{3-5}
    & & $2$ & $0$ &  if $\nu$ is even \\
\end{tabular}
$$
Similarly, the distribution of Hasse invariants arising from genera $G_\p$ above with fixed partial local genus symbols (when $\p\mid 2$ and $F_\p = \Q_2$) is given by 
$$
\begin{tabular}{c | c | c | c | c}
\text{Partial local genus symbols} & $\det_H(G_\p)$ & $c_\p =1$ & $c_\p =-1$ & Conditions\\
\hline
%% (\bar{2})
\multirow{2}{*}{$(\bar{2})$} & \multirow{2}{*}{$u$} & $0$ & $0$ & if $u \equiv 1_{(4)}$ \\
\cline{3-5}
    & & $0$ & $1$ &  if $u \equiv 3_{(4)}$ \\
\hline
%% (2) 
\multirow{2}{*}{$(2)$} & \multirow{2}{*}{$u$} & $1$ & $1$ & if $u \equiv 1_{(4)}$ \\
\cline{3-5}
    & & $1$ & $0$ &  if $u \equiv 3_{(4)}$ \\
\hline
%% (1,1)
%
$(1,1)$ & -- & $1$ & $1$ & --\\
\hline
%% (1;1) 
\multirow{2}{*}{$(1;1)$} & \multirow{2}{*}{$u$} & $1$ & $1$ & if $u \equiv 1_{(4)}$ \\
\cline{3-5}
    & & $2$ & $0$ &  if $u \equiv 3_{(4)}$ \\
\hline
%% (1 :: 1)
\multirow{2}{*}{$(1::1)$} & \multirow{2}{*}{$u \pi_\p^\nu$} & $2$ & $2$ & 
	if $\nu$ is odd or $u \equiv 1_{(4)}$ \\
\cline{3-5}
    & & $4$ & $0$ &  if $\nu$ is even and $u \equiv 3_{(4)}$ \\
\end{tabular}
$$
\end{lem}

\begin{proof}

{\bf Case 1:} When $\p\nmid 2$, the (modular) Jordan blocks are summands determined up to isomorphism by their determinant squareclass (and dimension and scale).  The Jordan block structure $J = (2)$ has one block and $Q \sim_{\Op} x^2 + uy^2$, so there is one genus for each allowed determinant $u$.  Here the Hasse invariant is $c_\p = (1,u)_\p= 1$.  
When $J = (1,1)$ then $Q \sim_{\Op} u_1x^2 + \pi_\p^\nu u_2 y^2$, where the $u_i\in \Op^\times$ can be freely chosen so that $u_1u_2 = u$ is fixed and $\nu$ is determined by the determinant.  This gives two forms, and their Hasse invariants are given by the Hilbert symbol 
$$
c_\p(Q) = (u_1, \pi_\p^\nu u_2)_\p 
= \cancel{(u_1, u_2)_\p} \cdot (u_1, \pi_\p)^\nu_\p 
= \leg{u_1}{\p}^\nu.
$$
When $\nu$ is even this is 1, but when $\nu$ is odd this takes both values $\pm1$ once.

\smallskip
{\bf Case 2:} Now suppose that $\p\mid 2$ and $F_\p = \Q_2$, so $\Op = \Z_2$ and we can use the ``train/compartment'' invariant theory of quadratic forms over $\Z_2$ in \cite[\textsection7.5, p381--2]{CS-book} to enumerate local genera.  This amount to decorating the partial local genus symbols (translated via Remark \ref{Rem:CS_PLGS_translation}) with ``signs'' and ``oddities'' to obtain a local genus symbol.  Note that in general a change of sign does not affect the reduction of the determinant (mod 4), and allows us to freely vary among these two squareclasses.
When $J = (\bar{2})$ then there are exactly two quadratic forms, having determinants 3 and 7 (mod 8), both with $c_\p = -1$.  When $J = (2)$ then there are three possible oddities and one choice of sign.  When $u \equiv 1$ (mod 4) we have two oddities have $c_\p =\pm1$, but when $u \equiv 3$ (mod 4) then we have one oddity with $c_\p = 1$.  

If there are no barred numbers in the partial local genus symbol $J$, then it corresponds to a diagonalizable quadratic form $Q \sim_{\Op} ax^2 + 2^\nu by^2$ for some $a,b \in \Op^\times$.  
Also all forms of this determinant are scalings of $Q$ by some $u \in \SqOp$, 
and we can easily compute their Hasse invariants using the formula
$$
c_\p(u Q) = (u a, u 2^\nu b)_\p 
= (u,u)_\p 
\cdot 
(u, 2^\nu)_\p 
\cdot 
(u, ab)_\p 
\cdot 
(a,2^\nu b)_\p.
$$

When $J = (1,1)$ we have four possible oddities $(\{0,2,4,6\})$ and two signs, and we see that the change of oddity and sign to preserve the determinant preserves $(u,u)_\p$, but reverses $(u, 2)_\p$, so we have $c_\p = \pm1$ once for each determinant.  
When $J = (1;1)$ then $\nu = 2$ so scaling to account for different oddities (with fixed sign) changes $c_\p$ by $(u,u)_\p \cdot (u,ab)_\p$, which is constant for all $u$ iff $ab \equiv u \equiv 3$ (mod 4).  By choosing $a\equiv 1$ (mod 8) we see that we always have at least one $c_\p = 1$, giving the $c_\p$-distribution above.
Finally when $J=(1::1)$ we have two choices of sign and four choices of unsigned oddities, giving four possible rescaling. When $\nu$ is odd or $ab \equiv 1$ (mod 4) we can use the previous cases to see that $c_\p$ will change sign, and so be equibistributed, otherwise all $c_\p$ values are 1. 
\end{proof}

\begin{lem}[$\p$-masses and local densities]
Given a primitive local genus $G_\p$ of non-degenerate binary quadratic forms with associated ideal $\I(\det_H(G_\p)) = \p^\nu$,  
the $\p$-mass $m_\p(G_\p)$ given by replacing $p$ by $q$ in \cite[eq (3), p263]{CS} is related to the inverse local density $\beta^{-1}_{G_\p, \p}(G_\p)$ by the formula
$$
%\text{Insert Formula here!}
\beta^{-1}_{G_\p, \p}(G_\p) = 2m_\p(G_\p) \cdot q^{-\frac{3\nu}{2} + 3 \ord_\p(2)},
$$
where either $\p\nmid 2$ or $\p\mid 2$ and $F_\p = \Q_2$.
\end{lem}

\begin{proof}
The formula \cite[\textsection12, pp281-2]{CS} adapted to $\p\nmid 2$ or where $F_\p = \Q_2$ gives 
$$
2 \beta_{Q,\p}(Q) m_\p(Q) = q^{\frac{1}{2}\textstyle (n+1)\cdot \ord_\p(\I(\det_G(Q)))},
$$
and the formula follows since $\p^\nu = \det_H(G_\p) = 2^n \cdot \det_G(G_\p)$.
\end{proof}

\begin{rem}
The $p$-mass formulas given in \cite[eq (3) and \textsection12]{CS} are still valid for any $\p\nmid 2$ because they have been independently derived in \cite{Pall}, and the argument given there holds for computing the (non-degenerate) local densities $\beta_{Q,\p}(Q)$ for any $\p\nmid 2$ by replacing the expression ``$p$'' there by either $\p$ or $q$ as appropriate.
\end{rem}

We now adopt a particularly convenient convention for defining the generic local densities when no normalized local genus exists (see Definition \ref{Def:generic_density_and_product}).

\begin{defn} \label{Def:generic_local_density_convention_at_n=2}
Suppose that $p=2$ splits completely in $F$, and $n=2$.  Then for normalized $\widetilde{S_\p} \in \SqFpInt$ where 
$\p\mid 2$ and $\widetilde{S_\p} \equiv 1_{(4)}$, we set 
$$
\beta^{-1}_{n=2; \p}(\widetilde{S_\p}) 
:= 2 \cdot \gamma_\p(\widetilde{S_\p})^{-1}
= \frac{2}{1 - \frac{\chi_{\widetilde{S_\p}}(\p)}{2}}
= \frac{2}{1 - \frac{1}{2}\leg{\widetilde{S_\p}}{2}},
$$
where $\gamma_\p(u) := 1 - \frac{\chi_{u}(\p)}{2}$, $\chi_{u}(\p) := \leg{-u}{\p}$, and $\leg{u}{2}$ is the extended Kronecker symbol defined as 
$$
\leg{u}{2} := 
\begin{cases}
1 & \qquad\text{if $u \equiv \pm 1_{(8)}$,} \\
-1 & \qquad\text{if $u \equiv \pm 3_{(8)}$,} \\
0 & \qquad\text{otherwise.} \\
\end{cases}
$$
In the proof of Theorem \ref{Thm:Explicit_A_and_B_for_n=2} we will see that this definition also gives $\beta^{-1}_{n=2; \p}(\widetilde{S_\p})$  when $\widetilde{S_\p} \equiv -1_{(4)}$, and it is this uniformity  motivates our choice here.
\end{defn}

%With these lemmas we are ready to
We now compute the normalized Euler factors at $\p\nmid 2$ and also at $\p\mid2$ where $F_\p =\Q_2$. 

\begin{thm}  \label{Thm:Explicit_A_and_B_for_n=2}
When $n=2$ and the prime $p=2$ splits completely in $F$, we have explicit rational functions for the 
``fixed unit'' local Dirichlet series described in Theorem \ref{Thm:Rationality_of_Euler_factors}, 
%Euler factors of the formal Dirichlet series $D_{A, {S_n}; n}(s)$ and $D_{B, {S_n}; n}(s)$ 
given by 
$$
\sum_{\nu=0}^\infty \frac{A^*_{\p; 2}(u\, \pi_\p^\nu)}{q^{\nu s}} = 
\begin{cases}
\frac{1 - \chi_u(\p) \, q^{-(s+2)}}{1 - q^{-(s+1)}} & \qquad \text{if $\p\nmid 2$,} \\
\frac{1 - \chi_u(\p) \, q^{-(s+3)}}{1 - q^{-(s+1)}} & \qquad \text{if $\p\mid 2$ and $F_\p = \Q_2$,} \\
\end{cases}
$$
and 
$$
\sum_{\nu=0}^\infty \frac{B^*_{\p; 2}(u\, \pi_\p^\nu)}{q^{\nu s}} = 
\begin{cases}
\frac{1 - \chi_u(\p) \, q^{-(2s+3)}}{1 - q^{-(2s+2)}} & \qquad \text{if $\p\nmid2$,} \\
0 & \qquad \text{if $\p\mid 2$, $F_\p = \Q_2$, and $u\equiv 1$ (mod $4\Op$),} \\
\frac{-1 + 2 q^{-(2s+2)} - \chi_u(\p) \, q^{-(2s+3)}}{1 - q^{-(2s+3)}} & \qquad \text{if $\p\mid 2$, $F_\p = \Q_2$, and $u\equiv 3$ (mod $4\Op$),} \\
\end{cases}
$$
where $\chi_u(\p) := \leg{-u}{\p}$
and where  $\textstyle \leg{\cdot}{\p}$ is defined as the non-trivial quadratic character on $\SqCl(k_\p^\times)$ when $\p\nmid 2$, and as the Kronecker symbol at $\p\mid2$ when $F_\p = \Q_2$ (i.e. which takes values 1 if $u\equiv \pm1_{(8)}$ and $-1$ if $u\equiv \pm3_{(8)})$.  
\end{thm}

\begin{proof}
For convenience, we let $\gamma_{\p}(u) := \textstyle{1 - \frac{\chi_u(\p)}{q}}$. 
From Lemma \ref{Lem:Count_binary_Jordan_and_PLGS},
when $\p\nmid 2$ there are two cases to consider and when $F_\p = \Q_2$ there are five cases to consider.  The normalized densities and Hasse invariant distributions for each of these cases, and local factor computations are given in the following tables:

\bigskip
%%%%%%%%%%%%%%%%%%

%%%%%%%%%%%%%%%%%%%%%%%%%%%%%%%
%%  Table 1 -- Summarize the Local Computations (n=2)   %%
%%%%%%%%%%%%%%%%%%%%%%%%%%%%%%%
\resizebox{0.98\hsize}{!}{
%\begin{table}[t]
%\caption{Local genus types and related invariants for primes $p>2$.} 
\centering
\hspace{-0.2in} 
\begin{tabular}{  | c |  c  | c |  c  |  c  | c  | c  | c | c |}
\multirow{2}{*}{prime $\p$} & \multirow{2}{*}{\text{\#}}   
    & Allowed & Jordan& \multirow{2}{*}{Cases }& \multicolumn{2}{c|}{\# of Hasse Invariants}
    & \multirow{2}{*}{$\p$-mass $m_\p$} & Normalized Densities
    %$\beta_{p, Q}(Q)$
\\
 &  & $\nu = \ord_\p(S)$ & Blocks & & $\hspace{.13in}c_\p = 1\hspace{.13in}$ & $c_\p = -1$ 
    & & $\gamma_\p(u) \cdot\beta_{Q, \p}^{-1}(Q)$ 
\\
 \hline
%% ===========
%% p > 2, Case 1
%% ===========
\multirow{3}{*}{$\p\nmid 2$} 
& 1 & \multirow{1}{*}{$\nu = 0$} & I${}_2$ 
	& -- & 1 & 0 & $\frac{1}{2}\gamma_\p(u)^{-1}$ & \multirow{1}{*}{1}\\
\cline{2-9}
%% ===========
%% p > 2, Case 2
%% ===========
 & \multirow{2}{*}{2} & \multirow{2}{*}{$\nu \geq 1$} & \multirow{2}{*}{I{}$_1 \oplus $I${}_1$} & $\nu$ odd 
    & 1 & 1 & \multirow{2}{*}{$\frac{1}{4}\,q^{\frac{\nu}{2}}$} 
    &\multirow{2}{*}{$\frac{1}{2}\, q^{-\nu} \cdot \gamma_\p(u)$} 
    \\
 \cline{5-7}
 & &  & & $\nu$ even  
    & 2  & 0 & &
    \\
% & &  &  
%    & ($\ve_2 = 1$)  & ($\ve_2 = -1$) & %$\frac{1}{2p^b}\(1 + \frac{c_p}{p} \)$ 
%    \\
\hline
\hline
%% ===========
%% p = 2, Case 1
%% ===========
\multirow{10}{.6in}{$\p\mid2$ and $F_\p = \Q_2$} 
 & \multirow{3}{*}{1} & \multirow{3}{*}{$\nu = 0$} & \multirow{3}{*}{II${}_2$}
 	& $u\equiv 1_{(4)}$ & 0 & 0 & \multirow{1}{*}{--} & \multirow{1}{*}{--} \\
\cline{5-9}
&  &  & & $u\equiv 3_{(8)}$ & \multirow{2}{*}{0} & \multirow{2}{*}{1} 
	& \multirow{2}{*}{$2^{-2} \gamma_\p(u)^{-1}$} & \multirow{2}{*}{2}\\
\cline{5-5} %\cline{8-9}
&  &  & & $u\equiv 7_{(8)}$ &  &  &  &  \\
\cline{2-9}
%% ===========
%% p = 2, Case 2
%% ===========
 & \multirow{2}{*}{2} & \multirow{2}{*}{$\nu = 2$} &  \multirow{2}{*}{I{}$_2$ } 
 	& $u\equiv 1_{(4)}$ & 1 & 1 & \multirow{1}{*}{$2^{-3}$} 
		& $2^{-2} \gamma_\p(u)$\\
\cline{5-9}
&  &  & & $u\equiv 3_{(4)}$ & 1 & 0 
	& $2^{-2}$ & $2^{-1} \gamma_\p(u)$\\
\cline{2-9}
%% ===========
%% p = 2, Case 3
%% ===========
 & \multirow{1}{*}{3} & \multirow{1}{*}{$\nu = 3$} &  \multirow{1}{*}{I{}$_1 \oplus $I${}_1$} 
 	& -- & 1 & 1 & \multirow{1}{*}{$2^{-\frac{5}{2}}$} & $2^{-3} \, \gamma_\p(u)$\\
\cline{2-9}
%% ===========
%% p = 2, Case 4
%% ===========
 & \multirow{2}{*}{4} & \multirow{2}{*}{$\nu = 4$} &  \multirow{2}{*}{I{}$_1 \oplus $I${}_1$} 
 	& $u\equiv 1_{(4)}$ & 1 & 1 & \multirow{2}{*}{$2^{-2}$} & \multirow{2}{*}{$2^{-4} \, \gamma_\p(u)$} \\
\cline{5-7}
&  &  & & $u\equiv 3_{(4)}$ & 2 & 0 &  & \\
\cline{2-9}
%% ===========
%% p = 2, Case 5
%% ===========
 & \multirow{2}{*}{5} & \multirow{2}{*}{$\nu \geq 5$} &  \multirow{2}{*}{I{}$_1 \oplus $I${}_1$} 
 	& $\nu$ odd or $u\equiv 1_{(4)}$ & 2 & 2 & \multirow{2}{*}{$2^{\frac{\nu}{2} - 5}$} & \multirow{2}{*}{$2^{-\nu-1} \, \gamma_\p(u)$} \\
\cline{5-7}
&  &  & & $\nu$ even and $u\equiv 3_{(4)}$ & 4 & 0 &  & \\
\hline
\end{tabular}
%\label{Tab:binary_p_invariants}
%\end{table}
}

\bigskip
%%%%%%%%%%%%%%%%%%

%%%%%%%%%%%%%%%%%%%%%%%
%%  Table 2 -- Give A and B terms (n=2)    %%
%%%%%%%%%%%%%%%%%%%%%%%
\resizebox{0.98\hsize}{!}{
%\begin{table}[t]
%\caption{Local genus types and related invariants for primes $p>2$.} 
\centering
\hspace{-0.2in} 
\begin{tabular}{  | c |  c  | c |  c  |  c  | c  | c  |}
\multirow{2}{*}{prime $\p$} & \multirow{2}{*}{\text{\#}}   
    & Allowed & Jordan & \multirow{2}{*}{Cases }
    &  \multirow{2}{*}{$A^*_{\p; n=2}(S)$} &  \multirow{2}{*}{$B^*_{\p; n=2}(S)$} 
\\
 &  & $\nu = \ord_\p(S)$ & Blocks & & &
\\
 \hline
%% ===========
%% p > 2, Case 1
%% ===========
\multirow{3}{*}{$\p\nmid2$} 
& 1 & \multirow{1}{*}{$\nu = 0$} & I${}_2$ 
	& -- & 1 & 1 \\
\cline{2-7}
%% ===========
%% p > 2, Case 2
%% ===========
 & \multirow{2}{*}{2} & \multirow{2}{*}{$\nu \geq 1$} & \multirow{2}{*}{I{}$_1 \oplus $I${}_1$} & $\nu$ odd 
    &\multirow{2}{*}{$q^{-\nu} \cdot \gamma_\p(u)$} 
    &\multirow{1}{*}{$0$} 
    \\
 \cline{5-5} \cline{7-7}
 & &  & & $\nu$ even  & & $q^{-\nu} \cdot \gamma_\p(u)$
    \\
% & &  &  
%    & ($\ve_2 = 1$)  & ($\ve_2 = -1$) & %$\frac{1}{2p^b}\(1 + \frac{c_p}{p} \)$ 
%    \\
\hline
\hline
%% ===========
%% p = 2, Case 1
%% ===========
\multirow{10}{.6in}{$\p\mid2$ and $F_\p = \Q_2$} 
 & \multirow{3}{*}{1} & \multirow{3}{*}{$\nu = 0$} & \multirow{3}{*}{II${}_2$}
 	& $u\equiv 1_{(4)}$ &  \multirow{1}{*}{--} & \multirow{1}{*}{--} \\
\cline{5-7}
&  &  & & $u\equiv 3_{(8)}$ & \multirow{2}{*}{$1$}  & \multirow{2}{*}{$-1$} \\
\cline{5-5}
&  &  & & $u\equiv 7_{(8)}$ &  & \\
\cline{2-7}
%% ===========
%% p = 2, Case 2
%% ===========
 & \multirow{2}{*}{2} & \multirow{2}{*}{$\nu = 2$} &  \multirow{2}{*}{I{}$_2$ } 
 	& $u\equiv 1_{(4)}$ & \multirow{2}{*}{$2^{-2} \gamma_\p(u)$} & 0\\
\cline{5-5}  \cline{7-7}
&  &  & & $u\equiv 3_{(4)}$ &  & $2^{-2} \gamma_\p(u)$\\
\cline{2-7}
%% ===========
%% p = 2, Case 3
%% ===========
 & \multirow{1}{*}{3} & \multirow{1}{*}{$\nu = 3$} &  \multirow{1}{*}{I{}$_1 \oplus $I${}_1$} 
 	& -- &    $2^{-3} \,  \gamma_\p(u)$ & 0 \\
\cline{2-7}
%% ===========
%% p = 2, Case 4
%% ===========
 & \multirow{2}{*}{4} & \multirow{2}{*}{$\nu = 4$} &  \multirow{2}{*}{I{}$_1 \oplus $I${}_1$} 
 	& $u\equiv 1_{(4)}$ & \multirow{2}{*}{$2^{-4} \,  \gamma_\p(u)$}  & 0 \\
\cline{5-5} \cline{7-7}
&  &  & & $u\equiv 3_{(4)}$ &  & $2^{-4} \,  \gamma_\p(u)$ \\
\cline{2-7}
%% ===========
%% p = 2, Case 5
%% ===========
 & \multirow{2}{*}{5} & \multirow{2}{*}{$\nu \geq 5$} &  \multirow{2}{*}{I{}$_1 \oplus $I${}_1$} 
 	& $\nu$ odd or $u\equiv 1_{(4)}$ &  \multirow{2}{*}{$2^{-\nu} \,  \gamma_\p(u)$} 
	& 0 \\
\cline{5-5} \cline{7-7}
&  &  & & $\nu$ even and $u\equiv 3_{(4)}$ &  & $2^{-\nu} \,  \gamma_\p(u)$ \\
\hline
\end{tabular}
%\label{Tab:binary_p_invariants}
%\end{table}
}

\noindent
(Note: Here our convention for the generic local density at $\p \mid 2$ requires an extra division by 2 for the normalized local densities before using them in the second table.)

\bigskip

When $\p\nmid2$ these tables give
\begin{align*}
\sum_{\nu \geq 0} A^*_{\p; 2}(u\, \pi_\p^\nu) X^\nu 
&=
1 +  \(1- \tfrac{\chi_u(\p)}{q}\) \sum_{\nu \geq 1} q^{-\nu} X^\nu 
=
%\\
%&=
1 +  \frac{(1- \frac{\chi_u(\p)}{q})\frac{X}{q}}{1 - \frac{X}{q}} 
=
%\\
%&=
 \frac{1- \frac{\chi_u(\p) X}{q^2}}{1 - \frac{X}{q}} \\
\end{align*}
and 
\begin{align*}
\sum_{\nu \geq 0} B^*_{\p; 2}(u\, \pi_\p^\nu) X^\nu 
&=
1 +  \(1- \tfrac{\chi_u(\p)}{q}\) \sum_{\substack{\nu \geq 2 \\ \nu \text{ even}}} q^{-\nu} X^\nu 
=
%\\
%&=
1 +  \frac{(1- \frac{\chi_u(\p)}{q})\frac{X^2}{q^2}}{1 - \frac{X^2}{q^2}} 
=
%\\
%&=
 \frac{1  - \frac{\chi_u(\p) X^2}{q^3}}{1 - \frac{X^2}{q^2}} 
\end{align*}
which give the desired formulas for $\p\nmid2$ after substituting $X := q^{-s}$.  

When $\p\mid2$ and $F_\p= \Q_2$ we have 
\begin{align*}
\sum_{\nu \geq 0} A^*_{\p^\nu, u} X^\nu 
%&
=
1 +  \(1- \tfrac{\chi_u(\p)}{q}\) \sum_{\nu \geq 2} q^{-\nu} X^\nu %\\
%&
=
1 +  \frac{(1- \frac{\chi_u(\p)}{q})\frac{X^2}{q^2}}{1 - \frac{X}{q}} %\\
%&
=
 \frac{1- \frac{\chi_u(\p) X}{q^3}}{1 - \frac{X}{q}} %\\
\end{align*}
and when $\p\mid 2$ and $u \equiv 3$ (mod $4\Op$) we have 
\begin{align*}
\sum_{\nu \geq 0} B^*_{\p; 2}(u\, \pi_\p^\nu) X^\nu 
&=
-1 +  \(1- \tfrac{\chi_u(\p)}{q}\) \sum_{\substack{\nu \geq 2 \\ \nu \text{ even}}} q^{-\nu} X^\nu 
=
%\\
%&=
-1 +  \frac{(1- \frac{\chi_u(\p)}{q})\frac{X^2}{q^2}}{1 - \frac{X^2}{q^2}} 
=
%\\
%&=
 \frac{-1 + \frac{2 X^2}{q^2} - \frac{\chi_u(\p) X^2}{q^3}}{1 - \frac{X^2}{q^2}} 
\end{align*}
which give the desired formulas for $\p\nmid2$ after substituting $X := q^{-s}$.  When $\p\mid 2$ and $u\equiv1$ (mod $4\Op$) then $B_{\p; 2}(u\, \pi_\p^\nu) = 0$ for all $\nu$.
\end{proof}

\begin{rem} \label{Rem:B_is_zero_for_nonsquare_ideals}
Notice that from our table of local computations in the proof of Theorem \ref{Thm:Explicit_A_and_B_for_n=2} that $B^*_{n=2}(S) = 0$ unless $\I(S)$ is a square and $S \equiv 3$ (mod $4\OF$)  In this case we have $A^*_{n=2}(S) = \pm B^*_{n=2}(S)$, where $\pm = (-1)^\tau$ and $\tau$ is the number of primes $\p\mid 2$ where $\p\nmid \I(S)$. 
%which only occurs for us for discriminant $-4$.
%If $\I(S)$ is a square then $B(S) = 0$ unless ....
\end{rem}

We now compute the total mass of totally definite quadratic lattices of any given determinant over number fields $F$ where $p=2$ splits completely. These results give an independent (global) way of computing the local series in Theorem \ref{Thm:Explicit_A_and_B_for_n=2}.

\begin{lem} \label{Lem:Unevaluated_mass_for_binary_lattices}
Suppose that $F$ is a totally real number field.  Then for every $S\in\SqAfInt$ for which 
%$\mathbf{Cls}(S, \vec{\sigma}_\infty; n = 2) \neq \emptyset$, 
there exists a totally definite $\OF$-valued rank $2$ quadratic $\OF$-lattice $L$ with $\det_H(L) = S$ and signature vector $\vec{\sigma}_\infty$, 
we have
$$
%\hspace{-.5in}
\sum_{
L \in \mathbf{Cls}^*(S, \vec{\sigma}_\infty; n=2)
%\substack{
%\text{Classes of primitive totally} \\
%\text{positive definite $\Z$-valued} \\
%\text{binary quadratic lattices $L$} \\
%\text{with discriminant $-S$}
%}
}
\frac{1}{|\Aut(L)|}
=
	\frac{
	\beta^{-1}_{n=2, \mathbf{f}}(\widetilde{S})
	\cdot 
	|\Delta_F|^\frac{1}{2}
%	L(1, \chi_t) 
	\cdot N_{F/\Q}(\I(S))^{\frac{3}{2}} }{2 \cdot (4\pi)^{[F:\Q]}}
%	\cdot 
\biggl[
	A^*_{n=2}(S) + \ve_\infty \cdot B^*_{n=2}(S) 
\biggr],
$$
where $\ve_\infty := (-1)^\al$ and  $\al$ is the number of archimedean places $v$ where the local signature $\sigma_v$ is negative definite.
\end{lem}

\begin{proof}
Since $V(2) = \frac{\pi}{2}$ in Corollary \ref{Cor:Mass_formula_result_for_globally_rational}, we have that
$$
M^*_{\ve_\infty; 2}(S)  
= 
\frac{(4\pi)^{[F:\Q]}}{|\Delta_F|^\frac{1}{2} N_{F/\Q}(\I(S))^{\frac{3}{2}}}
\cdot
%\hspace{-.3in}
\sum_{
Q \in \mathbf{Cls}^*(S, \sigma; 2)
}
\frac{1}{|\Aut(L)|}
$$
with $\ve_\infty = 1$.
When $\S = \emptyset$ in Theorem \ref{Cor:Mass_formula_result_for_globally_rational}, looking at the $S$-coefficient gives 
$$
M^*_{\ve_\infty; 2}(S)  
=
\tfrac{1}{2} \beta_{n=2, \mathbf{f}}^{-1}(\widetilde{S})
	\cdot \[A^*_{n=2}(S) +  \ve_\infty \cdot B^*_{n=2}(S)\]
$$
which proves the theorem.
\end{proof}

\begin{rem}
This issue of existence of global genera of lattices in Lemma \ref{Lem:Unevaluated_mass_for_binary_lattices} is equivalent to the local existence (at all $\p$) together with the condition that $S$ is globally rational.  The local existence question is discussed in Remark \ref{Rem:Local_genus_existence_when_2_splits_completely}, and gives the exact existence criterion when $p=2$ splits completely in $F$.
\end{rem}

\begin{lem} \label{Lem:Generic_local_product_when_n_is_two}
Suppose that $p=2$ splits completely in $F$.  Then with the conventions in Definition \ref{Def:generic_local_density_convention_at_n=2}, the generic density product can be evaluated as
$$
\beta^{-1}_{n=2, \mathbf{f}}(S)
=
2^{[F:\Q]} \cdot 
	\prod_\p \gamma_\p(\widetilde{S})^{-1}.
$$
\end{lem}

\begin{proof}
When $\p\nmid 2$ we have a unique local genus $G_\p$ with $\det_H(G_\p) = \widetilde{S}_\p$, and $\beta^{-1}_{{G_\p}, \p}(G_\p) = \gamma_\p(S)^{-1} := \frac{1}{1-\frac{\chi_\mathfrak{t}(\p)}{q}}$.  
Now suppose that $\p\mid2$, $F_\p = \Q_2$.  
When $\widetilde{S}_\p = -1 \in \SqCl((\Op/4\Op)^\times)$ there is also a unique genus $G_\p$ with $\det_H(G_\p) = \widetilde{S}_\p$, but here the doubled $\p$-mass $2m_\p = \frac{1}{4} \gamma_\p(S)^{-1}$, giving $\beta^{-1}_{{G_\p}, \p}(G_\p) = 2 \gamma_\p(S)^{-1}$. When $\widetilde{S}_\p = 1 \in \SqCl((\Op/4\Op)^\times)$ then there is no local genus $G_\p$ with $\det_H(G_\p) = \widetilde{S}_\p$, but following Definition \ref{Def:generic_local_density_convention_at_n=2} we define the normalized local density in this case to again be $2 \gamma_\p(S)^{-1}$.
Since $2$ splits completely in $F$, we have that $\beta^{-1}_{n=2, \mathbf{f}}(\widetilde{S}) = 2^{[F:\Q]} \prod_\p \gamma^{-1}_\p(S)$, which proves the lemma.
\end{proof}

\begin{thm}  Suppose that $F$ is totally real,  $p=2$ splits completely in $F$, and $\vec{\sigma}_\infty^+$ is the totally definite signature vector of rank 2 (i.e. $\sigma_v^+ = (2,0)$ for all $v\mid\infty$).  Then 
for every $S\in\SqAfInt$ where $\mathbf{Cls}^*(S, \vec{\sigma}_\infty^+; n = 2) \neq \emptyset$, we have 
$$
\sum_{
L \in \mathbf{Cls}^*(S, \vec{\sigma}_\infty^+; n = 2)
}
\frac{1}{|\Aut(L)|}
= 
\frac{\kappa(S) \cdot |\Delta_F|^\frac{1}{2} \cdot N_{F/\Q}(\I(S))^{\frac{1}{2}}}{2 \cdot (2\pi)^{[F:\Q]}}
\cdot
\prod_{\p\nmid \I(S)} \gamma_\p(S)^{-1}
$$
where %$\lambda = 1$ if $\I(S) =
$$
\kappa(S) :=
\begin{cases}
2 & \qquad \text{if $\I(S) = \square$, $\widetilde{S} \equiv 3_{(4)}$, and $\tau$ is even} \\
0 & \qquad \text{if $\I(S) = \square$, $\widetilde{S} \equiv 3_{(4)}$, and $\tau$ is odd} \\
1 & \qquad \text{otherwise,} \\
\end{cases}
$$
and $\tau := $ the number of primes $\p\mid 2$ with $\p\nmid \I(S)$.
\end{thm}

\begin{proof}
By Remark \ref{Rem:B_is_zero_for_nonsquare_ideals} when $\I(S) \neq \square$ or $\widetilde{S} \equiv 1_{(4)}$ we only need to consider $A^*_2(S)$  since there $B^*_2(S) = 0$.  
From the table in the proof of Theorem \ref{Thm:Explicit_A_and_B_for_n=2}, we see that $A^*_\p(S)$ is  $q^{-\nu_\p}$ where $\nu_\p := \ord_\p(\I(S))$, with a possible factor of $\gamma_\p(u)$ which appears iff $\p\mid \I(S)$.  
Combining these we have 
$$
A^*_{n=2}(S) = 
%	\[\prod_{\p\mid 2}\lambda_\p(S)\] \cdot
	\frac{1}{N_{F/\Q}(\I(S))} \cdot 
	\prod_{\p\mid \I(S)} \gamma_\p(\widetilde{S}).
$$
%where $\lambda(S)$ is given in ZZZZZZ.  
Thus by Lemma \ref{Lem:Generic_local_product_when_n_is_two} the product 
$$
\beta^{-1}_{n=2, \mathbf{f}}(\widetilde{S}) \cdot A^*_{n=2}(S) = 
%	\[\prod_{\p\mid 2}\lambda_\p(S)\] \cdot 
	\frac{2^{[F:\Q]}}{N_{F/\Q}(\I(S))} \cdot 
	\prod_{\p\nmid \I(S)} \gamma_\p(\widetilde{S})^{-1},
$$
and the result follows from Lemma \ref{Lem:Unevaluated_mass_for_binary_lattices}.

When $\I(S) = \square$ and $\widetilde{S} \equiv 1_{(4)}$ then from Remark \ref{Rem:B_is_zero_for_nonsquare_ideals} we have $B^*_{n=2}(S) = (-1)^\tau A^*_{n=2}(S)$, proving the cases where $\kappa(S) = 0$ and $2$.
\end{proof}

%%%%%%%%%%%%%%%%%%%%%%%
%%  The Analytic Class Number Formula  %%
%%%%%%%%%%%%%%%%%%%%%%%
\section{The analytic class number formula}

In this section we explain how to interpret our previous results about binary quadratic lattices in terms of class numbers of relative quadratic orders, by using Kneser's generalized Dedekind correspondence between quadratic lattices and ideal classes of quadratic extensions.
One interesting corollary of this formula is to recover the Dirichlet class number formula for CM extensions of totally real fields $F$ where $p=2$ splits completely (in $F$).    
This elucidates some comment of Siegel \cite[p11, pp124-5]{Siegel:1963vn} where he states that his general mass formula recovers Dirichlet's class number formula when applied to binary quadratic forms.

\begin{defn}
%Define the {\bf non-archimedean discriminant squareclass} for (relative) quadratic algebras.
Suppose that $R$ is a quadratic $\OF$-algebra in the sense of \cite[\textsection2, p407]{Kneser:1982kx}.  Then we define the its {\bf (non-archimedean) discriminant squareclass} $\Disc_{\OF}(R) \in \SqAfInt$ by requiring that the local squareclasses $\Disc_{\OF}(R)_\p$ are the discriminant squareclasses $(b^2 - 4c)(\Op^\times)^2$ of the local free quadratic algebras $R_\p \cong \Op[x]/(x^2 + bx + c)$. 
\end{defn}

\begin{defn}
We say that two signature vectors $\vec{\sigma}_\infty$ and $\vec{\sigma}'_\infty$ for $F$ of rank $n$ are {\bf locally similar} if for every real place $v$ of $F$ we have that $\sigma_v = (\sigma_{v, +}, \sigma_{v, -})$ is equal to either $(\sigma'_{v, +}, \sigma'_{v, -})$ or $(\sigma'_{v, -}, \sigma'_{v, +})$, which corresponds to the relation that the associated local quadratic spaces are similar for each $v\mid\infty$.
\end{defn}

%\bigskip
%\hrule
%\bigskip

To connect classes of binary quadratic forms with ideal classes in a relative quadratic extension, we translate some results of Kneser on composition laws for binary quadratic forms into our language.

\begin{lem}[Translating Kneser's Quadratic Lattices] \label{Lem:Kneser_G(C)_translation}
Suppose that $C$ is a quadratic $\OF$-algebra, and let $G(C)$ denote the group of projective rank two primitive quadratic $\OF$-lattices $(L, Q)$ which are rank one $C$-modules and satisfy the norm-compatibility condition $Q(c\cdot \x) = \Nm(c) \cdot Q(\x)$ for all $\x \in L$ and all $c \in C$, as described in \cite[\textsection6, p441]{Kneser:1982kx}. Then %$G(C)$ 
$$
G(C) = \bigsqcup_{\substack{\vec{\sigma}'_\infty \text{ is locally } \\ \text{ similar to } \vec{\sigma}_\infty(C)}}
\mathbf{Cls}^*(\textstyle{\det_H(C)}, \vec{\sigma}'_\infty; n=2)
$$
where $\det_H(C) \in \SqAfInt$ and $\vec{\sigma}_\infty(C)$ are respectively the non-archimedean Hessian determinant squareclass and signature vector of the binary quadratic $\OF$-lattice $C$ equipped with its (quadratic) norm form $\Nm_{C/\OF}$, and $\vec{\sigma}'_\infty$ runs over all signature vectors of rank $2$ which are definite at exactly the real places $v$ where $\vec{\sigma}_\infty(C)$ is definite.
\end{lem}

\begin{proof}
If $(L,Q) \in G(C)$ then by the norm-compatibility condition at each place $v$ we have the similarity condition $(L_v, Q_v) \cong_{\Ov} u_v \cdot (C_v, \Nm_{C_v / \Ov})$ where $u_v := Q(\x_v)$ for any $\x_v \in L_v$ that generates $L_v$ as a (free) rank one $C_v$-module.  At non-archimedean places $\p$, since the local determinant squareclass of an even rank quadratic lattice is unaffected by local unit scaling, this shows that $\det_H(L) = \det_H(C) \in \SqAfInt$.  At real places $v$ we have that local signatures $\sigma_v(C_v, \Nm_v) = (2,0)$ or $(1,1)$ when $(C_v, \Nm_v)$ is respectively definite or indefinite since $\Nm(1) = 1 > 0$.  Similarly, the local similarity of $L$ and $C$ at real places $v$ shows that $\sigma_v(L, Q)$ is definite $\iff \sigma_v(C, \Nm)$ is definite, so we have the inclusion $\subseteq$.

Conversely, suppose that $(L, Q) \in \mathbf{Cls}^*(S, \vec{\sigma}_\infty; n=2)$ for some choice of $S$ and $\vec{\sigma}_\infty$.  Then from \cite[p407, top]{Kneser:1982kx} we know that $L$ is a rank 2 module over its even Cifford algebra $C := C(L)$ satisfying the norm compatibility condition above, and $C$ is a quadratic $\OF$-algebra.  However since a quadratic $\OF$-algebra is determined locally up to isomorphism by its non-archimedean discriminant squareclass in $\SqAfInt$ (hence determined globally), and locally at all primes $\p$ the discriminant squareclass of $C$ must agree with $S$ (from the previous argument), we see that $C$ is independent of our choice of $L$.  (Also the previous argument now shows that $\vec{\sigma}_\infty$ is locally similar to $\vec{\sigma}_\infty(C)$.)  This shows the opposite inclusion $\supseteq$, proving the claim.
\end{proof}

%With the previous lemma, We can  establish an generalized version of Dedekind's correspondence

\begin{lem}[Generalized Dedekind Correspondence]  \label{Lem:Dedekind_correspondence}
Suppose that $K/F$ is a CM extension of number fields.  Then, with $G(C)$ as in Lemma \ref{Lem:Kneser_G(C)_translation} and $C = \OK$, we have 
$$
|G(\OK)| = \frac{h(\OK)}{h(\OF)} \cdot \frac{2^{[F:\Q]}}{Q_{K/F}},
$$
where $Q_{K/F} := [\OK^\times: \OF^\times \cdot (\OK^\times \cap \mu_\infty)]$ and $\mu_\infty$ denotes the group of roots of unity in $\bar{\Q}$.
\end{lem}

\begin{proof}  From Kneser \cite[p412, top]{Kneser:1982kx} we have the exact sequence
$$
\xymatrix{
\OK^\times \ar[r]^{\Nm}
	& \OF^\times \ar[r] 
	& G(\OK) \ar[r] 
	& \Pic(\OK) \ar[r]^{\Nm}
	& \Pic(\OF) \ar[r]
	& 1, 
}
$$
where the last entry follows from the surjectivity result \cite[Thrm 10.1, p184]{Washington:1982uq}, giving 
$$
|G(\OK)| = \frac{h(\OK)}{h(\OF)} \cdot |\OF^\times / \Nm_{K/F}(\OK^\times)|.
$$
To compute the size of this last group, notice that the the norm map gives an isomorphism 
$$
\xymatrix{
\OK^\times / (\OF^\times \cdot (\OK^\times \cap \mu_\infty)) \ar[r]^{\sim\qquad} 
	& \Nm_{K/F}(\OK^\times) / \Nm_{K/F}(\OF^\times),
}
$$
of groups of size $Q_{K/F}$, 
and that $\OF^\times / \Nm_{K/F}(\OF^\times) = \OF^\times / (\OF^\times)^2$ has size $2^{([F:\Q] - 1) + 1}$ by Dirichlet's unit theorem and because the only roots of unity in $F$ are $\{\pm 1\}$.  Combining these gives the desired formula.
\end{proof}

We now count automorphisms of binary lattices in $G(\OK)$ in preparation for generalizing Dirichlet's class number formula.

\begin{lem}[Computing Automorphisms] \label{Lem:binary_lattice_automorphisms}
Suppose that $K/F$ is a CM extension of number fields and that $(L, Q) \in G(\OK)$ as in Lemmas \ref{Lem:Kneser_G(C)_translation} and \ref{Lem:Dedekind_correspondence}.  Then $\Aut^+(L) = \mu_K := K^\times \cap \mu_\infty$ and $|\Aut(L)| = 2 \cdot |\mu_K|$.
\end{lem}

\begin{proof}
Since $\Aut(L)$ are exactly the automorphisms $\Aut(V,Q)$ of the ambient quadratic space $(V,Q)$ stabilizing $L$, we first determine $\Aut(V,Q)$.  Notice that the norm-compatibility condition ensures $L$ corresponds to a (not necessarily free) lattice in a scaled version of the quadratic space $(K, \Nm_{K/F})$, and that scaling a quadratic space does not affect its automorphisms.

By taking the (ordered) $F$-basis $\{1, \sqrt{\Delta}\}$ for $K$, (giving the matrix representation $\al := a + b \sqrt{\Delta} \leftrightarrow M_\al := \begin{bmatrix} a & b\Delta \\ b & a \end{bmatrix}$ and $\Nm_{K/F}(\al) = \det(M_\al)$) and explicitly solving for all $\gamma \in GL_2(F)$ satisfying
${}^t\gamma M \gamma = M$ with $M = \begin{bmatrix} 1 & 0 \\ 0 & -\Delta \end{bmatrix}$ and $\Delta \in F$ with $K = F(\sqrt{\Delta})$, we see that the (rational) automorphisms of the quadratic space have the form
$$
\Aut(K, \Nm_{K/F}) \cong \Gal(K/F) \times \{K^\times \mid \Nm_{K/F}(K^\times) = 1\}.
$$
Since for $\al \in K^\times$ we have $\det(M_\al) = \Nm_{K/F}(\al) = 1$, and the non-trivial Galois element $\sigma$ has $\det(\sigma) = -1$, we see that 
$$
\Aut^+(K, \Nm_{K/F}) = \{K^\times \mid \Nm_{K/F}(K^\times) = 1\}.
$$

Now suppose that $\gamma \in \Aut^+(L,Q)\subseteq \GL_2(F)$.  By Lemma \ref{Lem:Kneser_G(C)_translation} we know that $(L,Q)$ is totally definite, hence $|\Aut(L)| < \infty$.  Since $\gamma \cdot L = L$ we know that $\gamma \in \OK^\times$, but since $\gamma$ has finite order by Dirichlet's unit theorem we see that $\gamma \in \mu_K$.  Conversely, the norm compatibility condition shows that $\mu_K \subseteq \Aut^+(L)$, proving the first claim.

From the Kneser exact sequence at $\Pic(C)$, we see that the underlying $\OK$-modules $I$ of all $L \in G(\OK)$ are exactly those which (as ideals ) have $N_{K/F}(I) = (I \cdot \sigma(I)) \cap F = a\OF$ for some $a \in F^\times$, and so $I \cdot \sigma(I) = a\OK$ giving $\sigma(I) = I$ in $\Pic(\OK)$.  Therefore $\sigma \in \Aut(L)$, proving the second claim.
\end{proof}

We are now in a position to derive a relative version of Dirchlet's class number formula for CM extensions from our previous results.

\begin{thm}[Analytic Class Number Formula]     \label{Thm:Analytic_class_number_formula}
Suppose that $K/F$ is a CM extension of number fields and $S := \det_H(\OK, \Nm_{K/F}) \in \SqAfInt$, and that $p=2$ splits completely in $F$.  Then we have the analytic class number formula
$$
h(\OK) = 
	\frac{|\mu_K| \cdot h(\OF) \cdot |\Delta_F|^\frac{1}{2} \cdot Q_{K/F} \cdot N_{F/\Q}(\I(S))^{\frac{1}{2}}}
		{(2\pi)^{[F:\Q]}} \cdot 
	L_F(1, \chi_{K/F}),
%	\prod_{\p\nmid\I(S)} \gamma_\p(\widetilde{S})^{-1}.
$$
where $Q_{K/F} := |\OK^\times / \mu_K \cdot \OF^\times| \in \{1,2\}$ and $\chi_{K/F}$ is the non-trivial order 2 Hecke character over $F$ associated to the extension $K/F$ by class field theory.
When $F= \Q$, we have $h(\OF) = Q = 1$ and we recover Dirichlet's analytic class number formula for imaginary quadratic fields $K$.
\end{thm}

\begin{proof}
From Lemmas \ref{Lem:Dedekind_correspondence} and \ref{Lem:binary_lattice_automorphisms} we have that
$$
h(\OK) 
= 
\frac{h(\OF) \cdot Q_{K/F}}{2^{[F:\Q]}} \cdot |G(\OK)|
=
\frac{2 \cdot |\mu_K| \cdot h(\OF) \cdot Q_{K/F}}{2^{[F:\Q]}} \cdot \frac{|G(\OK)|}{|\Aut(L)|}
$$
for any $L \in G(\OK)$.  By Lemma \ref{Lem:Kneser_G(C)_translation} we can re-express this as a sum over proper masses of binary  quadratic lattices with totally definite signature vectors $\vec{\sigma}'_\infty$, giving
$$
h(\OK) 
=
\frac{2 \cdot |\mu_K| \cdot h(\OF) \cdot Q_{K/F}}{2^{[F:\Q]}} 
	\cdot \sum_{\substack{\text{totally} \\ \text{definite $\vec{\sigma}'_\infty$}}}
		\sum_{L \in \mathrm{Cls}(S, \vec{\sigma}'_\infty; n=2)} \frac{1}{|\Aut(L)|}.
$$
%{\bf (WARNING: We wanted the proper classes here, but instead we have the classes... so we need to do more work!  Let's keep going for now, and see how the ambiguous classes are needed!)}
%
Since $\vec{\sigma}'_\infty$ freely runs over all both the $(2,0)$ and $(0,2)$ local signatures $\sigma'_v$ at each archimedean place $v$, we see that when applying Lemmas \ref{Lem:Unevaluated_mass_for_binary_lattices} and \ref{Lem:Generic_local_product_when_n_is_two} to evaluate these sums the contribution from the $B$-series cancels out due to the variation of $\ve_\infty \in \{\pm1\}$, giving
$$
h(\OK) 
=
\frac{\cancel{2} \cdot |\mu_K| \cdot h(\OF) \cdot Q_{K/F}}{\cancel{2^{[F:\Q]}}} 
	\cdot \cancel{2^{[F:\Q]}} 
	\cdot  \frac{|\Delta_F|^\frac{1}{2} \cdot N_{F/\Q}(\I(S))^{\frac{1}{2}}}{\cancel{2} \cdot (2\pi)^{[F:\Q]}}  
	\cdot \prod_{\p\nmid\I(S)} \gamma_\p(\widetilde{S})^{-1}.
$$
Finally, since $-S$ is the fundamental discriminant squareclass for $\OK$ as a quadratic $\OF$-algebra, we have that the conductor of $\chi_{K/F}$ is $\I(-S) = \I(S)$, and so $L_F(1, \chi_{K/F}) = \prod_{\p \nmid \I(S)} \gamma_\p(\widetilde{S})^{-1}$.
\end{proof}

\begin{rem}[Cancellation of $B^*$-terms]
It is interesting to see that the Dirichlet series for $B^*$ will a priori cancel out whenever the signature is not totally indefinite (i.e. not indefinite at all real places), and even then, our explicit computations (Remark \ref{Rem:B_is_zero_for_nonsquare_ideals}) show that we only have a contribution from $B^*$ when $\I(S)$ is a square and $S \equiv 3$ (mod $4\OF$).
\end{rem}

\begin{rem}[Class numbers of orders]
Suppose that $R_K$ is some order in $\OK$, where $K/F$ is a CM extension and $p=2$ splits completely in $F$.   Then one can also establish an analytic class number formula for $h(R_K)$ analogous to Theorem \ref{Thm:Analytic_class_number_formula} for $h(\OK)$.  By comparing these formulas for $h(R_K)$ and $h(\OK)$ 
one can show the relation 
$$
h(R_K) = h(\OK) \cdot \prod_{\p\mid \frac{\I(S')}{\I(S)}} \(1 - \frac{\chi_S(\p)}{q} \)
$$
where $S = \det_H(\OK, N_{K/F})$ and $S' = \det_H(R_K, N_{K/F})$,
which agrees with the (more general) class number formula for orders given by Shimura in \cite[\textsection{12.5}, pp116-7]{Shimura_Clifford}.
\end{rem}

\begin{rem}[Class number formula from $L$-functions]
The analytic class number formula in Theorem \ref{Thm:Analytic_class_number_formula} also agrees with the formula arising from taking $\mathrm{res}_{s=1} \frac{\zeta_K(s)}{\zeta_F(s)}$ for CM extensions $K/F$, which states that 
$$
h(\OK) = \frac{h(\OF) \cdot |\Delta_K|^\frac{1}{2} \cdot Q_{K/F} \cdot |\mu_K|}
	{(2\pi)^{[F:\Q]} \cdot |\Delta_F|^\frac{1}{2}} 
		\cdot L_F(1,\chi_{K/F}).
$$
(Here the ratio of regulators is computed using \cite[Prop. 4.16, p41]{Washington:1982uq}.)  
To see this we apply the relative discriminant formula \cite[Cor 2.10, p202]{Neukirch:1999cr} to the tower of extensions $K/F/\Q$, giving 
$$
\Delta_K = (\Delta_F)^2 \cdot  N_{F/\Q}(\Delta_{K/F}),
$$
and notice that the relative discriminant ideal $\Delta_{K/F} = \I(-S) = \I(S)$.
\end{rem}

\begin{rem}[Indefinite forms and non-CM extensions]
One can also perform similar computations for indefinite binary forms to obtain an analytic class number formula (specializing to Dirichlet's class number formula for real quadratic fields when $F=\Q$) 
for non-CM quadratic extensions $K/F$, where $p=2$ splits completely in $F$.   This is somewhat more complicated since it would require one to normalize the symmetric space and regulator measures appropriately, and perform the relevant (possibly non-finite index) unit group computations for $K/F$ in that setting.
For simplicity here we only treat the totally definite case, which both illustrates 
how one would proceed in the more general case and shows the explicit
connection between our work and analytic class number formulas.
\end{rem}

%%%%%%%%%%%%%%%%%%%%%%%%%%%%
%%  Recover the growth of class numbers of orders  %%
%%%%%%%%%%%%%%%%%%%%%%%%%%%%

%\centerline{\bf Add the correspondence with class numbers of Quadratic Orders!}

%%%%%%%%%%%%%%%%%%%%%%%%%%%%%%%
%% -----------------------------------------------------------------------------
%%  End of Results for binary quadratic forms
%% -----------------------------------------------------------------------------
%%%%%%%%%%%%%%%%%%%%%%%%%%%%%%%

%%%%%%%%%%%%%%%%%%%%%%%%%%%%%%%%%%%%%%%%%%%%
%%%%%%%%%%%%%%%%%%%%%%%%%%%%%%%%%%%%%%%%%%%%
%%%%%%%%%%%%%%%%%%%%%%%%%%%%%%%%%%%%%%%%%%%%

\bibliographystyle{plain}	% (uses file "plain.bst")
\bibliography{myrefs}		% expects file "myrefs.bib"

\end{document}